\documentclass{amsart}

\usepackage{amsmath,amsfonts,amsthm,tikz,amssymb,hyperref}

\usepackage{mathabx} 

\usepackage[enableskew]{youngtab}

\newtheorem{theorem}{Theorem}[section]
\newtheorem{proposition}[theorem]{Proposition}
\newtheorem{lemma}[theorem]{Lemma}
\newtheorem{corollary}[theorem]{Corollary}

{\theoremstyle{definition}
\newtheorem{example}[theorem]{Example}
\newtheorem{remark}[theorem]{Remark}
}

\numberwithin{equation}{section}
\numberwithin{table}{section}

\newcommand{\Sym}{\textit{Sym}}     
\newcommand{\U}[1]{U_{#1}}          
\newcommand{\UU}[1]{\U{\ss[#1 X]}}  
\newcommand{\D}[1]{D_{#1}}          
\newcommand{\DD}[1]{\D{\ss[#1 X]}}  
\newcommand{\K}[1]{K_{#1}}          
\newcommand{\KK}[1]{\K{\ss[#1 X]}}  
\newcommand{\KB}[1]{\overline{K}_{#1}} 
\newcommand{\QQ}{\mathbb{Q}}             

\renewcommand{\ss}{\sigma}              
\newcommand{\scalar}[2]{\left\langle\, #1 \, \middle| \, #2 \, \right\rangle} 
\newcommand{\noc}[1]{\#\mathrm{corners}(#1)}    
\newcommand{\plus}[1]{\widehat{#1}} 
\newcommand{\minus}[1]{\widecheck{#1}}
\newcommand{\add}[1]{#1^+} 
\newcommand{\remove}[1]{#1^-}
\newcommand{\addremove}[1]{{#1^{\mp}}}
\newcommand{\addrestrict}[2]{#1^{+#2}}
\newcommand{\addcomplement}[2]{#1^{+#2^c}}

\author[Briand]{Emmanuel Briand}
\address{Departamento de Matem\'atica Aplicada I, Escuela T\'ecnica Superior de Ingenier\'ia Inform\'atica,
Avda.\ Reina Mercedes, S/N,
41012 Sevilla, Espa\~na}
\email{ebriand@us.es}

\author[McNamara]{Peter R. W. McNamara}
\address{Department of Mathematics, Bucknell University, Lewisburg, PA 17837, USA}
\email{peter.mcnamara@bucknell.edu}

\author[Orellana]{Rosa Orellana}%
\address{Dartmouth College, Mathematics Department, 6188 Kemeny Hall, Hanover, NH 03755, USA} \email{rosa.c.orellana@dartmouth.edu}%

\author[Rosas]{Mercedes Rosas}
\address{Departamento de \'Algebra, Facultad de Matem\'aticas, Universidad de Sevilla, Avda.\ Reina Mercedes, Sevilla, Espa\~na}
\email{mrosas@us.es}

\begin{document}

\dedicatory{\centerline{Dedicated to Ira Gessel on the occasion of his retirement.}}


\title[operators on symmetric functions]{Commutation and normal ordering for operators on symmetric functions}

\begin{abstract}
We study the commutation relations and normal ordering between families of operators on symmetric functions.  These operators can be naturally defined by the operations of multiplication, Kronecker product, and their adjoints.   As applications we give a new proof of the skew Littlewood--Richardson rule and  prove an identity about the Kronecker product with a skew Schur function.   
\end{abstract}

\maketitle

\section{Introduction}

Due to the their connection to representation theory, Schubert calculus, and  their beautiful combinatorial description, Schur functions are ubiquitous in algebraic combinatorics.  For this reason, it not surprising that identities  involving Schur functions greatly improve the understanding of these subjects

There are two important products naturally defined on symmetric functions:  the ordinary  and the   Kronecker product.  The well-known Littlewood-Richardson coefficients are the structure coefficients for the ordinary product, while the  elusive Kronecker coefficients are the structure coefficients for the  the Kronecker product.
They naturally define  linear operators on symmetric functions.

Let $f$ be a symmetric function. Then,  define  $U_f$ to be  the operator  ``multiplication by $f$'', and $K_f$ be  ``Kronecker multiplication by $f$.''  Explicitly, the operators act on a symmetric function $g$ by
\[
\U{f}(g)=f g, \qquad \K{f}(g)=f \ast g.
\]
The Hall inner product allow us to define the adjoint of $U_f$, which is denoted by $D_f$, and sometimes called the skewing operator.  With respect to this inner product $K_f$ is self-adjoint.  We will also consider another intriguing  operator related to the Kronecker product,  $\KB{\lambda}$,  defined on the Schur basis as follows: Let $\lambda$ be a partition and $g$ be any homogeneous symmetric function of degree $n$, then
\[
\KB{\lambda}(g)=s_{(n-|\lambda|,\lambda)}\ast g\,,
\] 
where $(n-|\lambda|,\lambda)=(n-|\lambda|,\lambda_1,\lambda_2,\ldots)$ is a sequence of integers. This sequence  is not always decreasing. To deal with this issue, we define the Schur function   $s_{(n-|\lambda|,\lambda)}$ by means of the Jacobi--Trudi formula:
$
s_{(\alpha_1,\alpha_2,\ldots,\alpha_N)}=\det(h_{\alpha_{i}+j-i})_{i,j=1\ldots N}.$
This determinant coincides with the Schur function $s_{\alpha}$ when $\alpha$ is weakly decreasing (\emph{i.e.,} is a partition) but makes sense even when $\alpha$ is not. 
Once the operators $\KB{\lambda}$ are defined, the definition is extended by linearity to $\KB{f}$ for any symmetric function $f$. 
%

%
In this paper we study identities involving four families of operators on the ring of symmetric functions,  $\Sym$: $\U{\lambda}$, $\D{\lambda}$, $\K{\lambda}$, and $\KB{\lambda}$.  Here the index $\lambda$ is a partition and indicates that the operator is product by the Schur function $s_\lambda$.  We consider the following questions:
Given any pair of them.  Can we establish  the commutator relations between them? Given a word involving them, how can we put it in a normal form? Is this expression unique?
Is it possible to express some of them in terms of the other ones?

For a motivating example, let us look at the operators $\U{(1)}$ and $\D{(1)}$. They are well-known to satisfy the commutator relation
$
\D{(1)}\U{(1)} = \U{(1)}\D{(1)} +1.$ That is  Leibniz's rule for multiplication, when  $\U{(1)}$ is multiplication by $x$, and $\D{(1)}=\frac{\partial}{\partial x}$. This identity is the defining relation for the algebra of Weyl, and the building identity for Stanley's theory of   differential posets,
  \cite{Stanley-diffposet}.  

Our main result is the following theorem that gives beautiful commutation relations four each pair of operators. 
\begin{theorem}\label{thm:main}
For any partitions $\alpha$ and $\beta$ we have the following identities
(where $\lambda$, $\tau$ and $\nu$ each run over the set of all partitions).
\begin{align}
\D{\beta}\U{\alpha}&=
\sum_{\lambda} \U{{\alpha/\lambda}} \D{{\beta/\lambda}} \label{eqDbetaUalpha}\\
\U{\alpha}\D{\beta}&=\sum_{\lambda} (-1)^{|\lambda|} \D{{\beta/\lambda'}} \U{{\alpha/\lambda}}\label{eqUalphaDbeta}\\
\K{\beta} \U{\alpha} &=\sum_{\lambda} \U{s_{\beta/\lambda} \ast s_{\alpha}} \K{\lambda} \label{KU2}\\
\D{\alpha}\K{\beta} &=\sum_{\lambda} \K{\lambda} \D{s_{\beta/\lambda} \ast s_{\alpha}}   \label{DK2}\\
\KB{\beta} \U{\alpha}&=
\sum_{\tau,\nu}   \U{(s_{\beta/\nu} \ast s_{\tau}) s_{\alpha/\tau}} \KB{\nu} \label{KKKU2}\\
\D{\alpha} \KB{\beta} &=
\sum_{\tau,\nu} \KB{\nu}  \D{(s_{\beta/\nu} \ast s_{\tau}) s_{\alpha/\tau}}\label{DKKK2}
\end{align}
\end{theorem}
To answer our  set of questions, and in particular, to prove these identities, we associate a generating series to each of our operators. Then using known operations of formal series we are able to obtain a uniform and elegant method to tackle them.   More precisely,  let $P$ be any of operators $\U{}$, $\D{}$, $\K{}$ and $\KB{}$. We associate to  $P$  a formal series of operators
$
\sum_{\lambda} s_{\lambda}[A] P_{\lambda}.
$ We  call this series the  Schur generating series of $P$. 
The Schur generating series of $P$ defines a linear map that sends any symmetric function $g \in \Sym(X)$  to the expression
$
\sum_{\lambda} s_{\lambda}[A]  P_{\lambda}(g).
$
The operator $P_\lambda$ can be recovered from this series using the scalar product  on $\Sym$.


%
%
%
%

The first part of our paper concludes with the expansion of the identities in Theorem 1.1 in the Schur basis, and are summarized in Theorem \ref{thm:main-cor}. These expansions give  normal ordering relations for the operators $U, D, K,$ and $ \KB{}$.

An important question remains. Are such expressions unique?  We consider this question in  Proposition \ref{finite expansions}, where  we show that 
finite expansions with respect to the pairs $(U,D)$, $(D,U)$, $(U,K)$, $(K,D)$, $(U,\KB{})$ and $(\KB{},D)$ are unique.
In contrast, observe that expansions with respect to  $(K,U)$ or to $(D,K)$ are not unique. For instance we have the relation $\K{p_2} \U{p_1} =0$, that is straightforwardly equivalent to the relation $\K{2}\U{1} = \K{1,1}\U{1}$. Taking adjoints, we have the relation $\D{1}\K{2} = \D{1} \K{1,1}$.


The following charming and well-known identity $
\KB{(1)}=\U{(1)}\D{(1)}-1$  describes a  relation between the operators $\KB{}, U,$ and $D$.
 It translate to one of the few known cases where it is possible to give a combinatorial description for a Kronecker product. 
We  extend this identity to similar expression for  the Kronecker product of an arbitrary Skew Schur function with $s_{(n-1,1)}$, and for other simple families of Schur functions (hooks and two--row shapes).  

Remarkably, the previous identity shows that $\KB{(1)}$ can be rewritten it in terms of the much simpler operators $\U{(1)}$ and $\D{(1)}$.
In Proposition \ref{KBf in U D} we vastly generalize this observation, and show that for any symmetric function $f$ the operator $\KB{f}$ lies in the subalgebra of $\operatorname{End}(\Sym)$ generated by the operators $U_g$ and $D_g$.

Dealing with naturally defined objects, like our families of operators, you are bound to recover some classical results. A testimony of the elegance of this approach is that both Foulkes's and Littlewood's identities can be easily derived from Theorem \ref{thm:main-cor}.  However, a  new identity of the same nature is also obtained, as described  in the following table.  This is discussed in Section \ref{sec:commutations}.

\begin{table}[htb]
\begin{center}
\begin{tabular}{llc}
Product  & Coproduct                       & equivalent identities \\[2mm] \hline
Ordinary, $\cdot$ & Adjoint of the ordinary product & \eqref{eqDbetaUalpha} and \eqref{TrueFoulkes} \\
Kronecker, $\ast$ & Adjoint of the ordinary product & \eqref{KU2} and  \eqref{TrueLittlewood}\\ 
Ordinary, $\cdot $ & Adjoint of the Kronecker product & \eqref{DK2} and \eqref{Similar}
\end{tabular}
\end{center}
\caption{The three bialgebra structures on $\Sym$ and the corresponding identities $\partial_x \circ \mu  = \mu \circ \partial_{\Delta(x)}$.}\label{table pro copro}
\end{table}
We finish our work with  two combinatorial applications of our identities. The first one is a proof of the \emph{skew Littlewood--Richardson Rule}, a combinatorial rule that gives the product of two skew Schur functions as a linear combination of skew Schur functions, based on counting Young tableaux (Theorem~\ref{eqskewLR}). This rule was conjectured in \cite{AssafMcNamara}, and proved in \cite{LamLauveSottile}. Our proof relies on the normal ordering relation that decomposes the products $\U{\alpha} \D{\beta}$ as linear combinations of products of the form $\D{\beta/\lambda'} \U{\alpha/\lambda}$. It generalizes the algebraic proof given by Thomas Lam of the skew Pieri Rule (a particular case of the skew Littlewood--Richardson Rule) in the appendix of \cite{AssafMcNamara}. Indeed, Lam's proof relies on the same normal ordering relation, in the special case of $\beta$ having only one part. 

The second application exploits our normal ordering relation for the products $\KB{1} D_{\lambda}$. We extend the combinatorial rule for the expansion in the Schur basis of the Kronecker product of $s_{(n-1,1)}$ with a Schur function,  to the Kronecker product of $s_{(n-1,1)}$ with any skew Schur function (Theorem \ref{thm:skewcorners}).  Additionally, we give a different, combinatorial, proof of this result.

\section{Statement of the results and Historical context}\label{sec:commutations}

The main result of the paper is a list of six normal ordering relations,  stated in two equivalent ways in Theorems \ref{thm:main} and \ref{thm:main-cor}. These relations will be proved in Section \ref{sec:proofs}.  
%

\begin{theorem}\label{thm:main-cor} For any partitions $\alpha$ and $\beta$ we have the following identities.
\begin{align*}
\D{\beta}\U{\alpha}&=
\sum_{\mu,\nu} 
\left(\sum_{\lambda}  c^{\alpha}_{\lambda,\mu} \;c^{\beta}_{\lambda,\nu}\right) 
\U{\mu}\D{\nu}\\
\U{\alpha}\D{\beta}
&=
\sum_{\mu,\nu} 
\left(\sum_{\lambda}  (-1)^{|\lambda|} c^{\alpha}_{\lambda,\mu} \; c^{\beta}_{\lambda',\nu}\right) 
\D{\nu}\U{\mu}\\
\K{\beta} \U{\alpha} &=
\sum_{\mu,\nu} 
\left(\sum_{\lambda} g_{\alpha,\lambda,\mu} \; c^{\beta}_{\lambda,\nu} \right) 
\U{\mu}\K{\nu}\\
\D{\alpha}\K{\beta} &=
\sum_{\mu,\nu} 
\left(\sum_{\lambda} g_{\alpha,\lambda,\mu} \; c^{\beta}_{\lambda,\nu} \right) 
\K{\nu}\D{\mu}\\
\KB{\beta} \U{\alpha}&=
\sum_{\mu,\nu} 
\left( \sum_{\lambda, \sigma, \tau, \theta}
g_{\lambda,\tau,\theta} \; c^{\beta}_{\lambda,\nu} \; c^{\alpha}_{\tau,\sigma} \; c_{\theta,\sigma}^{\mu} \right) \U{\mu} \KB{\nu}\\
\D{\alpha} \KB{\beta} &=
 \sum_{\mu,\nu} 
\left( \sum_{\lambda, \sigma, \tau, \theta}
g_{\lambda,\tau,\theta} \; c^{\beta}_{\lambda,\nu} \; c^{\alpha}_{\tau,\sigma} \; c_{\theta,\sigma}^{\mu} \right) \KB{\nu} \D{\mu}
\end{align*}
\end{theorem}

Identities \eqref{eqDbetaUalpha} and \eqref{KU2} are avatars of well-known identities of  Foulkes and Littlewood. 
 Indeed, if we apply the operators in \eqref{eqDbetaUalpha} and \eqref{KU2} to the Schur function $s_{\gamma}$ we get
\begin{align}
D_{\beta}(s_{\alpha}s_{\gamma})&=\sum_{\lambda} s_{\alpha/\lambda} \D{\beta/\lambda}(s_{\gamma}),\label{Foulkes}\\
s_{\beta}\ast (s_{\alpha} s_{\gamma}) &=\sum_{\lambda} (s_{\beta/\lambda} \ast s_{\alpha}) (s_{\lambda} \ast s_{\gamma}).\label{equ:littlewood}
\end{align}
By linearity, we can replace $s_{\alpha}$ and $s_{\gamma}$ with arbitrary symmetric functions $f$ and $g$. Also if we  expand $s_{\beta/\lambda}=\sum_{\mu} c^{\beta}_{\lambda, \mu} s_{\mu}$, we obtain
\begin{align}
\D{\beta}(fg)&=\sum_{\lambda,\mu} c^{\beta}_{\lambda,\mu} \D{\lambda}(f) \D{\mu}(g),  \label{TrueFoulkes}\\
s_{\beta}\ast (fg) &=\sum_{\lambda, \mu} c^{\beta}_{\lambda, \mu} (s_{\mu} \ast f)(s_{\lambda} \ast g).
\label{TrueLittlewood}
\end{align}
Formula \eqref{TrueFoulkes} was obtained by Foulkes (\cite[\S3.b]{FoulkesDifferentialOperators} , also mentioned in \cite[I.\S5 Ex.~25.(d)]{MacdonaldBook}, while \eqref{TrueLittlewood} is due to Littlewood (\cite[Theorem III]{Littlewood1956}, see also \cite[I.\S7 Ex.~23.(c)]{MacdonaldBook}). 
An expression similar to \eqref{TrueFoulkes} and \eqref{TrueLittlewood} can be derived the same way from \eqref{DK2}. This is:
\begin{equation}\label{Similar}
D_{\beta}(f \ast g) = \sum_{\lambda} g_{\beta, \lambda, \mu} D_{\lambda}(f) \ast D_{\mu}(g).
\end{equation}
The similarity between \eqref{TrueFoulkes}, \eqref{TrueLittlewood} and \eqref{Similar} has a nice explanation. It is provided by J.--Y. Thibon for \eqref{TrueFoulkes} and \eqref{Similar}, see \cite[p.554]{ThibonCoproductCR}, \cite[Proposition 6.4]{ThibonHopf}, and \cite{ThibonSlides}, but applies as well for \eqref{TrueLittlewood}. The explanation is as follows: let $B$ be a bialgebra with product $\odot$ and coproduct $\Delta$. The coproduct $\Delta$ induces a product $\mu$ on the dual space $B^*$. For any $x \in B$ or in $B \otimes B$, let $\partial_x$ be the adjoint of the $\odot$--product by $x$. Then, for any $x \in B$:
\begin{equation}\label{Thibon}
\partial_x \circ \mu  = \mu \circ \partial_{\Delta(x)} 
\end{equation}
See the aforementionned references by J.--Y. Thibon for a proof.

Consider $\Sym$ with a product that is either the ordinary product or the Kronecker product, and a coproduct that is either the adjoint of the ordinary product, or the adjoint of the ordinary coproduct. This gives four possibilities, but only three of them make $\Sym$ a bialgebra. The three identities  \eqref{TrueFoulkes}, \eqref{TrueLittlewood} and \eqref{Similar} are obtained by applying \eqref{Thibon} to these three bialgebra structures, with $x=s_{\beta}$. See table \ref{table pro copro}. Note that for any symmetric function $f$, the operator $K_f$ is its own adjoint.

Thibon also relates in \cite{ThibonSlides} Identity \eqref{TrueLittlewood} to Mackey's formula in group theory (see \cite[18.15]{Hazewinkel}). 
Let us mention that Cummins \cite[identity (13)]{Cummins} also derives an identity very close to \eqref{DK2}.

Formula \eqref{eqUalphaDbeta} can be stated as 
\begin{equation}
s_{\alpha} s_{\gamma/\beta}=\sum_{\lambda} (-1)^{|\lambda|} \D{\beta/\lambda'}(s_{\alpha/\lambda}s_{\gamma}). \label{ReverseFoulkes}
\end{equation}
Formula \eqref{eqUalphaDbeta} happens to be closely related  to the \emph{Skew Littlewood--Richardson Rule}, see section \ref{SkewLRrule}. In \cite[Lemma~1.1]{LamLauveSottile} Lam, Lauve, and Sottile obtain a more general version of Formula \eqref{eqUalphaDbeta} valid for arbitrary pairs of dual Hopf algebras.

As mentioned in the introduction, Ira Gessel, \cite{Gessel},  established special cases of \eqref{eqDbetaUalpha} and \eqref{eqUalphaDbeta} when the Schur functions are indexed by one-row or one-column shapes. He showed that 
\begin{eqnarray*}
&&D_n U_m = \sum_i U_{m-i} D_{n-i}\,, \\
&&U_m D_n= D_n U_m - D_{n-1} U_{m-1}\,,\\
&&D_{(1^n)} U_m= U_m D_{(1^n)} + U_{m-1} D_{(1^{n-1})}.
\end{eqnarray*}

Since $\K{\beta}$ and $\KB{\beta}$ are self-adjoint, and $U_{\alpha}$ and $D_{\alpha}$ are adjoint of each other, \eqref{DK2} and \eqref{DKKK2} are 
obtained from \eqref{KU2} and \eqref{KKKU2} respectively by taking adjoints.

An interesting and elegant way of stating some of the results of Theorem \ref{thm:main} is in terms of commutators. 
\begin{corollary}\label{commutators}
For any two partitions $\alpha$ and $\beta$, we have
\begin{align*} 
[\D{\beta}, \U{\alpha}]
&= \sum_{\lambda \neq (0)} \U{\alpha/\lambda} \D{\beta/\lambda} = \sum_{\lambda \neq (0)} (-1)^{|\lambda|-1} \D{\beta/\lambda'} \U{\alpha/{\lambda}}\,,\\
[\KB{\beta}, \U{\alpha}]&=
\sum_{(\tau,\nu) \neq ((0),\beta)}   \U{(s_{\beta/\nu} \ast s_{\tau}) s_{\alpha/\tau}} \KB{\nu}\,, \\
[\D{\alpha}, \KB{\beta}] &=
\sum_{(\tau,\nu) \neq ((0),\beta)} \KB{\nu}  \D{(s_{\beta/\nu} \ast s_{\tau}) s_{\alpha/\tau}}\,,
\end{align*}
where $(0)$ denotes the empty partition.
\end{corollary}


\section{Proof of the identities}\label{sec:proofs}


In this section, we prove all six identities of Theorem \ref{thm:main}.   We begin with an overview of the method of proof.   Let $P$ be any of the families of operators $\U{}$, $\D{}$, $\K{}$ and $\KB{}$. 
These operators act on $\Sym=\Sym(X)$, the symmetric functions in some alphabet $X$.  Introduce an auxiliary alphabet $A$. The Schur generating series of $P$ is defined as
\[
\sum_{\lambda} s_{\lambda}[A] P_{\lambda} .
\]
This can be interpreted as the linear map that sends any symmetric function $g \in \Sym(X)$  to the expression
\[
\sum_{\lambda} s_{\lambda}[A]  P_{\lambda}(g).
\]
For each of the four operators under consideration, the effect of the Schur generating function operator is described nicely by means of operations on alphabets (Lemma \ref{lemma:effects}). The identities of Theorem \ref{thm:main} are derived at the level of  generating series. The result is then recovered by extracting coefficients by means of the appropriate scalar products.

\subsection{Preliminary: operations on alphabets}

 Let $X=\{x_1, x_2, \ldots\}$ be the underlying alphabet for the symmetric functions in $\Sym$. Any infinite alphabet $A$  gives rise to a copy $\Sym(A)$ of $\Sym$. In this copy, the corresponding scalar product will be denoted by $\scalar{}{}_A$ and the element corresponding to $f\in\Sym$, by $f[A]$. 
Accordingly, the scalar product $\scalar{}{}$ of $\Sym$ and elements $f\in\Sym$ will be denoted sometimes by $\scalar{}{}_X$ and $f[X]$. 

If $A$ and $B$ are two alphabets, the tensor product $\Sym(A) \otimes \Sym(B)$ is endowed with the induced scalar product $\scalar{}{}_{A,B}$\,. 

Both the  Kronecker product $\ast$ and the adjoint $D_f$ of the operator of multiplication by a symmetric function $f$ will be only considered with respect to $\Sym=\Sym(X)$.

Given a morphism of algebras $A$ from $\Sym$ to some commutative algebra $\mathcal{R}$, it will be convenient to write it as 
 $f \mapsto f[A]$ (rather than $f \mapsto A(f)$) for any \emph{morphism of algebras}  and consider it as a ``specialization at the virtual alphabet $A$.'' 

Since the power sum symmetric functions $p_k$ ($k \geq 1$) generate $\Sym$ and are algebraically independent, the map 
\begin{equation}\label{bijection}
A \mapsto (p_1[A], p_2[A], \ldots)
\end{equation}
is a bijection from the set of all morphisms of algebras from $\Sym$ to $\mathcal{R}$ 
to the set of infinite sequences of elements from $\mathcal{R}$. This set of sequences is endowed with its operations of component-wise sum and product, and multiplication by a scalar. The bijection \eqref{bijection} is used to lift these operations to the set of morphisms from $\Sym$ to $\mathcal{R}$. This defines expressions like $f[A+B]$ and $f[AB]$, where $f$ is a symmetric function and $A$ and $B$ are two ``virtual alphabets,'' and more general expressions $f[P(A,B,\ldots)]$ where $P(A,B,\ldots)$ is a polynomial in several virtual alphabets $A$, $B$ \ldots with coefficients in the base field. Note that, by definition, for any power sum $p_k$ ($k \geq 1$), virtual alphabets $A$ and $B$, and scalar $z$,
\[
p_k[A+B]=p_k[A]+p_k[B], \qquad p_k[AB]=p_k[A] \cdot p_k[B], \qquad p_k[zA]=z \; p_k[A].
\]
In our calculations below, the morphism $f \mapsto f[1]$ will appear: it is the specialization at $x_1=1$, $x_2=0$, $x_3=0$ \ldots, sending each $p_k$ to $1$. The morphism $f \mapsto f[X]$ is just the identity of $\Sym$. The morphism $f \mapsto f[X^{\perp}] = f^{\perp}=D_f$ associates to $f$ the adjoint of the operator "multiplication by $f$".

Let $\ss$ be the generating series for the complete homogeneous symmetric functions $h_n$,  meaning
$\ss=\sum_{n=0}^{\infty} h_n$\,, where $h_0=1$. Recall from \cite[I.\S2]{MacdonaldBook} that we also have 
\[
\ss=\exp\left(\sum_{k=1}^{\infty} \frac{p_k}{k}  \right)=\sum_{\lambda} \frac{p_{\lambda}}{z_{\lambda}}
\]
where the last sum is carried over all partitions $\lambda$.

We will make use of the following known identities, whose proofs we include for the sake of completeness.
\begin{lemma} \label{properties:ss}
Let $A$ and $B$  be any two alphabets, and $f$ and $g$ be any two symmetric functions. Then we have the following identities.
\begin{align}
&\ss[A+B]=\ss[A] \ss[B]&&  \label{factorization}\\
&\ss[AB]=\sum_{\lambda} s_{\lambda}[A] s_{\lambda}[B] && \text{ (Cauchy Identity)} \label{Cauchy}\\
&\ss[-AB]=\sum_{\lambda} (-1)^{|\lambda|} s_{\lambda'}[A] s_{\lambda}[B] && \label{dual Cauchy}\\
&\D{\ss[AX]}(f[X])=f[X+A] && \label{sum}\\
&\ss[AX] \ast f[X] =f[AX] && \label{product}\\
&\scalar{f[AX]}{g[X]}_X=(f\ast g)[A] && \label{scalar}\\
&\scalar{\ss[AB]}{g[B]}_B=g[A] && \text{ (Reproducing Kernel)}\label{reproducing}
\end{align}
\end{lemma}

\subsection{Generating series}

Let $P$ be any of the families of operators $U$, $D$, $K$ and $\KB{}$. Introduce an auxiliary alphabet $A$ and the following generating series for $P$:
\[
\sum_{\lambda} s_{\lambda}[A] P_{\lambda}.
\]
We use the linearity of $f \mapsto P_f$ to simplify this expression as follows:
\begin{align*}
\sum_{\lambda} s_{\lambda}[A] P_{\lambda} 
= \sum_{\lambda} s_{\lambda}[A]  P_{s_\lambda[X]} 
= P_{\sum_{\lambda} s_{\lambda}[A] s_\lambda[X] } 
= P_{\ss[AX]}
\end{align*}
by the Cauchy Identity \eqref{Cauchy}.

Note that any operator $P_{\lambda}$ can be recovered from the generating series with a coefficient extraction by means of a scalar product:
\[
P_{\lambda}(f) = \scalar{P_{\ss[AX]}(f[X])}{s_{\lambda}[A]}_A\,.
\]

The generating series $P_{\ss[AX]}$ also acts linearly on symmetric functions. The following lemma describes the effect of all four generating series $\U{\ss[AX]}$, $\D{\ss[AX]}$, $\K{\ss[AX]}$ and $\KB{\ss[AX]}$.

\begin{lemma}\label{lemma:effects}
Let $f[X]$ be any symmetric function.
\begin{align}
\U{\ss[AX]}(f[X]) &= \ss[AX] \cdot f[X] \label{effectU}\\
\D{\ss[AX]}(f[X]) &= f[X+A] \label{effectD} \\
\K{\ss[AX]}(f[X]) &= f[AX]  \label{effectK} \\
\KB{\ss[AX]}(f[X]) &= \ss[-A] \cdot f[X(A+1)]\label{effectKbar}
\end{align} 
\end{lemma}

\begin{proof}
Equation \eqref{effectU} is straightforward. Equation \eqref{effectD} is \eqref{sum}. Equation \eqref{effectK} is \eqref{product}. Let us prove \eqref{effectKbar}. 
For any symmetric functions $f$ and $g$,
$
\KB{f}(g)=\Gamma_1 f \ast g
$,
where $\Gamma_1$ is the vertex operator:
\begin{equation}\label{Gamma-VO}
\Gamma_1=\left( \sum_{i=0}^{\infty} U_{(i)}\right) \left( \sum_{j=0}^{\infty} (-1)^j \D{(1^j)}\right)
\end{equation}

 We will make use of the following identity (see \cite[\S3]{ScharfThibonWybourne}. For a combinatorial approach to this identity see \cite{RosasGamma} ):
\begin{equation}\label{vertex}
\Gamma_1 f = \ss[X] f[X-1].
\end{equation}
Therefore, we have 
{\allowdisplaybreaks
\begin{align*}
\KB{\ss[AX]}(f)
&=\left(\Gamma_1 \; \ss[AX]\right) \ast f[X] && \\
&=\left(\ss[X] \; \ss[A(X-1)]\right) \ast f[X] && \\
&=\ss[X+A(X-1)] \ast f[X] && \text{by \eqref{factorization},}\\
&=\ss[X(A+1)-A] \ast f[X] && \\
&=\ss[-A] \; \left(\ss[X(A+1)] \ast f[X] \right)&& \text{by \eqref{factorization} again,}\\
&=\ss[-A] \cdot f[X(A+1)] && \text{by \eqref{product}.}\\
\end{align*}
}
\end{proof}

\subsection{Operators $U$ and $D$}

Let us prove the first two identities in Theorem \ref{thm:main}, which are \eqref{eqDbetaUalpha},
\begin{align*}
\D{\beta} \; \U{\alpha} &=
\sum_{\lambda} \U{{\alpha/\lambda}} \D{{\beta/\lambda}}\,, 
\end{align*}
and \eqref{eqUalphaDbeta},
\begin{align*}
\U{\alpha} \; \D{\beta}&=
\sum_{\lambda} (-1)^{|\lambda|} \D{{\beta/\lambda'}} \U{{\alpha/\lambda}}\,. 
\end{align*}

To this aim, we first establish commutation relations between the generating series of the operators $U$ and $D$.
\begin{lemma}
Let $A$ and $B$ be two alphabets. We have
{\allowdisplaybreaks
\begin{align}
\DD{B}\UU{A}&=\ss[AB] \UU{A}\DD{B},  \label{eqDU}\\
\UU{A}\DD{B}&=\ss[-AB] \DD{B}\UU{A}.\label{eqUD}
\end{align}
}
\end{lemma}
\begin{proof}
The second of these identities is obtained straightforwardly from the first after noting that $\ss[-AB]$ is the inverse of $\ss[AB]$ (see \eqref{factorization}). 

Let us prove \eqref{eqDU} using  Lemmas~\ref{properties:ss} and~\ref{lemma:effects}. We have, for any symmetric function $f[X]$,
{\allowdisplaybreaks
\begin{align*}
\DD{B} \UU{A} (f[X]) &= \DD{B} (\ss[AX] f[X]) && \\
&= \ss[A(X+B)] f[X+B] && \text{by \eqref{effectD},}\\
&= \ss[AB]\ss[AX]f[X+B] && \text{by \eqref{factorization},}\\
&= \ss[AB]\ss[AX]\D{\ss[BX]}(f[X]) && \text{by \eqref{effectD},}\\
&= \ss[AB]\UU{A}\DD{B}(f[X]).&&
\end{align*}
}
\end{proof}

\begin{proof}[Proof of \eqref{eqDbetaUalpha}]
We will use that, since $\U{}$ and $\D{}: \Sym \rightarrow \operatorname{End}(\Sym)$ are morphisms of algebras, we can write, for any symmetric function $f$, that $\U{f}=f[U]$ and $\D{f}=f[D]$. In particular, for the generating series, we have $\U{\ss[AX]}=\ss[AU]$ and $\D{\ss[BX]}=\ss[BD]$.

In \eqref{eqDU}, the operator $\D{\beta}\U{\alpha}$ is the coefficient of $s_{\alpha}[A]s_{\beta}[B]$ in the expansion in the Schur basis of $\ss[AB] \UU{A}\DD{B}$, which is extracted by performing the scalar product with  $s_{\alpha}[A]s_{\beta}[B]$.
Thus
{\allowdisplaybreaks
\begin{align*}
\D{\beta}\U{\alpha}
&=\scalar{\ss[AB] \UU{A} \DD{B}}{s_{\alpha}[A] s_{\beta}[B]}_{A,B} && \\
&=\sum_{\lambda} \scalar{s_{\lambda}[A] s_{\lambda}[B] \UU{A} \DD{B}}{s_{\alpha}[A] s_{\beta}[B]}_{A,B} && \text{by \eqref{Cauchy},}\\
&=\sum_{\lambda} \scalar{s_{\lambda}[A] \UU{A}}{s_{\alpha}[A] }_A \scalar{s_{\lambda}[B] \DD{B}}{s_{\beta}[B]}_{B} && \\
&=\sum_{\lambda} \scalar{\UU{A}}{s_{\alpha/\lambda}[A] }_A \scalar{\DD{B}}{s_{\beta/\lambda}[B]}_{B} && \\
&=\sum_{\lambda} \scalar{\ss[AU]}{s_{\alpha/\lambda}[A] }_A \scalar{\ss[BD]}{s_{\beta/\lambda}[B]}_{B} && \\
&=\sum_{\lambda} s_{\alpha/\lambda}[U]  s_{\beta/\lambda}[D] && \text{ by \eqref{reproducing},}\\
&=\sum_{\lambda} \U{\alpha/\lambda}  \D{\beta/\lambda}. && 
\end{align*}
}
\end{proof}
Identity \eqref{eqUalphaDbeta} is derived from \eqref{eqUD} analogously.


\subsection{Operators $U$ and $K$}


We now turn to the proof of \eqref{KU2}:
\[
\K{\beta} \U{\alpha} =\sum_{\lambda} \U{s_{\beta/\lambda} \ast s_{\alpha}} \K{\lambda} 
\]
(The identity \eqref{DK2} follows straightforwardly from \eqref{KU2} by taking adjoints).

Again we consider first commutation relations for the generating series of the operators $K$ and $U$.
\begin{lemma}
Let $A$ and $B$ be two alphabets.  We have
\begin{equation}\label{KKUU}
\KK{B} \UU{A}= \UU{AB} \KK{B}. 
\end{equation}
\end{lemma}
\begin{proof}
We have, for any symmetric function $f$,
{\allowdisplaybreaks
\begin{align*}
\KK{B} \UU{A} (f[X]) &= \KK{B}(\ss[AX] f[X])&&\\
&= \ss[ABX] f[BX] && \text{by \eqref{effectK},}\\
&= \UU{AB}(f[BX]) && \\
&= \UU{AB}\KK{B}(f[X])&& \text{by \eqref{effectK}.}
\end{align*}
}
\end{proof}

\begin{proof}[Proof of \eqref{KU2}]
We now get $\K{\beta}\U{\alpha}$ from \eqref{KKUU} by extracting the coefficient of $s_{\beta}[B] s_{\alpha}[A]$ in its expansion in terms of Schur functions:
{\allowdisplaybreaks
\begin{align*}
\K{\beta} \U{\alpha}  
&= \scalar{\UU{AB} \KK{B}}{s_{\beta}[B] s_{\alpha}[A]}_{A,B} && \\
&= \sum_{\lambda} \scalar{\UU{AB} s_{\lambda}[B]}{s_{\beta}[B] s_{\alpha}[A]}_{A,B} \K{\lambda}&& \\
&= \sum_{\lambda} \scalar{\UU{AB}}{s_{\beta/\lambda}[B] s_{\alpha}[A]}_{A,B} \K{\lambda}&& \\
&= \sum_{\lambda} \scalar{ \scalar{\UU{AB}}{s_{\alpha}[A]}_A}{s_{\beta/\lambda}[B]}_{B} \K{\lambda}&& \\
&= \sum_{\lambda} \scalar{ \scalar{\ss[ABU]}{s_{\alpha}[A]}_A}{s_{\beta/\lambda}[B]}_{B} \K{\lambda}&& \\
&= \sum_{\lambda} \scalar{ s_{\alpha}[BU]}{s_{\beta/\lambda}[B]}_{B} \K{\lambda}&& \text{by \eqref{reproducing},}\\
&= \sum_{\lambda} (s_{\alpha}\ast s_{\beta/\lambda})[U] \K{\lambda}&& \text{by \eqref{scalar},}\\
&= \sum_{\lambda} \U{s_{\alpha}\ast s_{\beta/\lambda}} \K{\lambda}.&& 
\end{align*}
}
\end{proof}


\subsection{Operators $U$ and $\KB{}$}

We now proceed to proving \eqref{KKKU2}:
\[
\KB{\beta} \U{\alpha}=
\sum_{\tau,\nu}   \U{(s_{\beta/\nu} \ast s_{\tau}) s_{\alpha/\tau}} \KB{\nu}\,. 
\]
(The identity \eqref{DKKK2} is deduced by taking adjoints).  Again, we first consider commutation relations for the generating series of the families of operators involved.

\begin{lemma}
Let $A$ and $B$ be two alphabets. We have
\begin{equation}\label{GKbarU}
\KB{\ss[BX]} \UU{A} = \UU{A(B+1)} \KB{\ss[BX]}.
\end{equation}
\end{lemma}
\begin{proof}
We have, for any symmetric function $f$, 
{\allowdisplaybreaks
\begin{align*}
\KB{\ss[BX]}\UU{A}(f)
&=\KB{\ss[BX]}(\ss[AX] f[X])&&\\ 
&=\ss[-B]\ss[AX(B+1)] f[X(B+1)]&& \text{by \eqref{effectKbar}, }\\
&=\ss[AX(B+1)] \ss[-B] f[X(B+1)]&&\\
&=\ss[AX(B+1)] \KB{\ss[BX]}(f)&& \text{by \eqref{effectKbar} again, }\\
&=\UU{A(B+1)} \KB{\ss[BX]}(f).&&
\end{align*}
}
\end{proof}

\begin{proof}[Proof of \eqref{KKKU2}]
From \eqref{GKbarU} we extract the term $\KB{\beta}\U{\alpha}$ by taking scalar product with $s_{\alpha}[A]\;s_{\beta}[B]$. This yields
\[
\KB{\beta}\U{\alpha} 
=\scalar{\UU{A(B+1)} \KB{\ss[BX]}}{s_{\alpha}[A] s_{\beta}[B]}_{A,B}.
\]
Expanding in the scalar product the generating function $\KB{\ss[BX]}$, we get
\[
\KB{\beta}\U{\alpha}
=\sum_{\nu}\scalar{\UU{A(B+1)} s_{\nu}[B]}{s_{\alpha}[A] s_{\beta}[B]}_{A,B} \KB{\nu}.
\]
This simplifies as follows:
{\allowdisplaybreaks
\begin{align*}
\KB{\beta}\U{\alpha}
&=\sum_{\nu}\scalar{\scalar{ \UU{A(B+1)}}{s_{\alpha[A]}}_A s_{\nu}[B]}{s_{\beta}[B]}_{B} \KB{\nu}&&  \\
&=\sum_{\nu}\scalar{\scalar{ \ss[A(B+1)U]}{s_{\alpha[A]}}_A s_{\nu}[B]}{s_{\beta}[B]}_{B} \KB{\nu}&&  \\
&=\sum_{\nu}\scalar{s_{\alpha[(B+1)U]} s_{\nu}[B]}{s_{\beta}[B]}_{B} \KB{\nu}&& \text{by \eqref{reproducing}, } \\
&=\sum_{\nu}\scalar{s_{\alpha[BU+U]} s_{\nu}[B]}{s_{\beta}[B]}_{B} \KB{\nu}&&  \\
&=\sum_{\nu}\scalar{\sum_{\tau} s_{\tau}[BU] s_{\alpha/\tau}[U]}{s_{\beta/\nu}[B]}_{B} \KB{\nu}&& \\
&=\sum_{\nu,\tau}\scalar{s_{\tau}[BU]}{s_{\beta/\nu}[B]}_{B} s_{\alpha/\tau}[U]\KB{\nu}&& \\
&=\sum_{\nu,\tau} (s_{\tau}\ast s_{\beta/\nu})[U] s_{\alpha/\tau}[U]\KB{\nu}&& \text{by \eqref{scalar},}\\
&=\sum_{\nu,\tau} \U{(s_{\tau} \ast s_{\beta/\nu})s_{\alpha/\tau}}\KB{\nu}\,.&& 
\end{align*}
}
\end{proof}

In Section \ref{SkewKronecker} we present as an application a combinatorial rule for the Kronecker
product of any skew Schur function by $s_{(n-1,1)}$. 
Other interesting particular cases of \eqref{KKKU2} correspond to the cases 
when $\lambda=(k)$ (Kronecker product with a two-row shape) and $\lambda=(1^k)$ (Kronecker product with a hook), where we get:
\[
\KB{(k)} \U{\alpha} =
\sum_{j=0}^k \left(\sum_{\rho \vdash k-j} \U{\alpha/\rho} \U{\rho}\right) \KB{(j)},\qquad
\KB{(1^k)} \U{\alpha} =
\sum_{j=0}^k \left(\sum_{\rho \vdash k-j} \U{\alpha/\rho} \U{\rho'}\right) \KB{(1^j)}\,.
\]
Setting $n=|\alpha|$ and $m=|\gamma|$, the same identities can be stated as:
\begin{align*}
s_{(n+m-k,k)}\ast(s_{\alpha}s_{\gamma}) =
\sum_{j=0}^k \left(\sum_{\rho \vdash k-j} s_{\alpha/\rho} s_{\rho} \right)( s_{\gamma}\ast s_{(m-j,j)}),\\
s_{(n+m-k,1^k)} \ast (s_{\alpha}s_{\gamma}) =
\sum_{j=0}^k \left(\sum_{\rho \vdash k-j} s_{\alpha/\rho} s_{\rho'}\right) (s_{\gamma} \ast s_{(m-j,1^j)}).
\end{align*}
The terms of the form $\sum_{\rho \vdash q} s_{\alpha/\rho}  s_{\rho}$ and $\sum_{\rho \vdash q} s_{\alpha/\rho}  s_{\rho'}$ that appear in these identities can take the following alternative forms, as a direct consequence of Littlewood's Identity \eqref{equ:littlewood}:
\[
\sum_{\rho \vdash q} s_{\alpha/\rho} s_{\rho} =s_{\alpha} \ast h_{(n-q,q)}, \quad \text{ and likewise }
\sum_{\rho \vdash q} s_{\alpha/\rho} s_{\rho'}  =s_{\alpha} \ast (h_{n-q}e_q).
\]


\section{Uniqueness of expansions}\label{uniqueness}

The identities in Theorem \ref{thm:main} and Corollary \ref{thm:main-cor} express some operators as linear combinations of operators $\U{\mu} \D{\nu}$, $\D{\nu} \U{\mu}$, $\U{\mu} \K{\nu}$, etc.  Are such expressions unique?

To answer this question, we associate to any pair $P$, $Q$ of linear maps from $\Sym$ to $\operatorname{End}(\Sym)$ a generating series depending on four independent alphabets $X$, $A$, $B$, $T$. This generating series is
\[
\sum_{\alpha, \beta,\lambda} P_{\alpha}(Q_{\beta}(s_{\lambda}[X])) s_{\alpha}[A] s_{\beta}[B] s_{\lambda}[T].
\]
It can be written more concisely as
\[
P_{\ss[AX]} Q_{\ss[BX]} (\ss[XT]).
\]
We also associate to the pair $P$, $Q$ the linear map $\Phi_{P,Q}$ from $\Sym(A) \otimes_Q \Sym(B)$ to the set of formal series of the form $\sum_{\alpha, \beta} a_{\alpha,\beta} s_{\alpha}[X] s_{\beta}[T]$, defined on simple tensors by:
\[
\Phi_{P,Q}( f[A] g[B]) = \scalar{P_{\ss[AX]} Q_{\ss[BX]} (\ss[XT])}{ f[A] g[B]}_{A,B}
\]

Let us say that \emph{finite expansions with respect to $(P,Q)$ are unique} if, for any $M \in \operatorname{End}(\Sym)$, there is at most one expansion
\[
 M = \sum_{\alpha, \beta} a_{\alpha, \beta} P_{\alpha} Q_{\beta}.
\]
That is, finite expansions with respect to $(P,Q)$ are unique when the operators $P_{\alpha} Q_{\beta}$, for $\alpha$ and $\beta$ partitions, are linearly independent.

\begin{proposition}\label{finite expansions}
Finite expansions with respect to the pairs $(U,D)$, $(D,U)$, $(U,K)$, $(K,D)$, $(U,\KB{})$ and $(\KB{},D)$ are unique.
\end{proposition}

\begin{example}
In contrast, observe that expansions with respect to  $(K,U)$ or to $(D,K)$ are not unique. For instance we have the relation $\K{p_2} \U{p_1} =0$, that is straightforwardly equivalent to the relation $\K{2}\U{1} = \K{1,1}\U{1}$. Taking adjoints, we have the relation $\D{1}\K{2} = \D{1} \K{1,1}$.
\end{example}

Proposition \ref{finite expansions} will be proved using the following lemma.
\begin{lemma}\label{coordinate_independent}
Let $P$ and $Q$ of linear maps from $\Sym$ to $\operatorname{End}(\Sym)$. 
\begin{enumerate}
\item Finite expansions with respect to $(P,Q)$ are unique if and only if $\Phi_{P,Q}$ is injective.
\item Let $M \in \operatorname{End}(\Sym)$. 
\begin{itemize}
\item The operator $M$ is in the linear span of the operators $P_f Q_g$, for $f$, $g \in \Sym$, if and only if 
$M(\ss[XT])$ lies in the image of $\Phi_{P,Q}$. 
\item If $M(\ss[XT]) = \Phi_{P,Q}(F)$ then $M$ is the image of $F$ under the linear map defined on the Schur basis by  $s_{\alpha}[A] s_{\beta}[B] \mapsto P_{\alpha} Q_{\beta}$.
\end{itemize}
\end{enumerate}
\end{lemma}

\begin{proof}
We start with a computation: let $M \in \operatorname{End}(\Sym)$. We have:
\begin{align*}
M = \sum_{\alpha, \beta} a_{\alpha, \beta} P_{\alpha} Q_{\beta}  
&\Leftrightarrow \text{ for any partition $\lambda$, } M(s_{\lambda}[X])  = \sum_{\alpha, \beta} a_{\alpha, \beta} P_{\alpha} Q_{\beta}  (s_{\lambda}[X]) \\
&\Leftrightarrow  \sum_{\lambda} M(s_{\lambda}[X]) s_{\lambda}[T] = \sum_{\alpha, \beta, \lambda} a_{\alpha, \beta} P_{\alpha} Q_{\beta}  (s_{\lambda}[X]) s_{\lambda}[T] \\
&\Leftrightarrow  M(\ss[XT])  = \sum_{\alpha, \beta} a_{\alpha, \beta} P_{\alpha} Q_{\beta}  (\ss[XT]) \\
&\Leftrightarrow  M(\ss[XT])  = \Phi_{P, Q}\left(\sum_{\alpha,\beta} a_{\alpha, \beta} s_{\alpha}[A] s_{\beta}[B]\right).
\end{align*}
This proves (2), since the linear span of the operators $P_f Q_g$, for $f$, $g \in \Sym$, is also the linear span of the operators $P_{\alpha} Q_{\beta}$, for $\alpha$ and $\beta$ partitions.

To obtain (1), take $M=0$ in the above equivalence.
\end{proof}

\begin{proof}[Proof of proposition \ref{finite expansions}]
For each of the pairs  $(U,D)$, $(D,U)$, $(U,K)$ and $(U,\KB{})$, we compute the corresponding generating  series by means of Lemma \ref{lemma:effects}. Next we deduce a description of the corresponding map $\Phi$ to show it is injective. The uniqueness of finite expansions follows then from (1) in Lemma \ref{coordinate_independent}.

For $(U,D)$, the detail of the calculation is as follows. The generating series is
\begin{align*}
U_{\ss[AX]} D_{\ss[BX]} (\ss[XT]) 
&=U_{\ss[AX]}  (\ss[(X+B)T], \\
&=\ss[AX] \ss[(X+B)T],  \\
&=\ss[AX+XT+BT].
\end{align*}

We deduce from this a formula for $\Phi_{U,D}$. Let $f$ and $g$ be any symmetric functions. We have
\begin{align*}
\Phi_{U,D}(f[A] g[B]) 
&= \scalar{\ss[AX+XT+BT]}{f[A] g[B]}_{A,B}&\\
&=\scalar{\ss[AX]\ss[XT]\ss[BT]}{f[A] g[B]}_{A,B}&\\
&=\ss[XT ]\scalar{\ss[AX]\ss[BT]}{f[A] g[B]}_{A,B}&\\
&=\ss[XT ]\scalar{\ss[AX]}{f[A]}_A \scalar{\ss[BT]}{g[B]}_{B}&\\
&=\ss[XT ] f[X] g[T] & \text{ by \eqref{reproducing}.}
\end{align*}
This shows that $\Phi_{U,D}$ is injective, since the series $\ss[XT]$ is invertible.

For the other three pairs, we skip the details of the calculations.

The generating series for $(D,U)$ is:
\[
D_{\ss[BX]} U_{\ss[AX]} (\ss[XT]) = \ss[(A+T)(B+X)].
\]
The map $\Phi_{D,U}$ is
\[
f[A] g[B] \mapsto \ss[XT] \ss[X^{\perp} T^{\perp}] (f[X] gT] )    
\]
The map $\Phi_{D,U}$ is injective, since the series $\ss[XT]$ is invertible, and the operator $\ss[X^{\perp} T^{\perp}]$ is invertible as well (its inverse is $\ss[- X^{\perp} T^{\perp}]$).

The generating series for $(U,K)$ is
\[
U_{\ss[AX]} K_{\ss[BX]} (\ss[XT]) = \ss[B(X+A)T].
\]
The map $\Phi_{U,K}$ is
\[
f[A] g[B] \mapsto  f[X] g[XT]
\]
To check that $\Phi_{U,K}$ is injective, post--compose it with the specialization of $T$ at $T/X$: the map obtained is $f[A] g[B] \mapsto  f[X] g[T]$, which is injective. Thus $\Phi_{U,K}$ is injective.

The generating series for $(U,\KB{})$ is
\[
U_{\ss[AX]} \KB{\ss[BX]} (\ss[XT]) = \ss[AX+B(XT-1)+XT].
\]
The map $\Phi_{U,\KB{}}$ is
\[
f[A] g[B] \mapsto     \ss[XT] f[X] g[XT-1]
\]
This map $\Phi_{U,\KB{}}$ is injective. Indeed, post--composing first with the product with the inverse of $\ss[XT]$, and next by the specialization of $T$ at $(T+1)/X$, yields the map $f[A] g[B] \mapsto  f[X] g[T]$, which is injective.

This proves the uniqueness of finite expansions with respect to $(U,D)$, $(D,U)$, $(U,K)$ and $(U,\KB{})$. The uniqueness of finite expansions with respect to $(K,D)$ and $(\KB{},D)$
 is obtained by taking adjoints.
\end{proof}

\begin{remark}
The uniqueness of finite expansions with respect to $(U, D)$ and $(D,U)$ can alternatively be proved by switching to the basis of power sums. The algebra generated by the operators $U_g$ and $D_g$ is also generated by $1$, the $U_{p_k}$ and the $D_{p_k}$ ($k>0$). The following maps define an isomorphism between this algebra and, firstly, the bosonic creation and annihilation operator algebra (this appears for instance in \cite{JimboMiwa}) and, secondly, the Weyl algebra in infinitely many generators. 
\begin{align*}
\U{p_k} &\mapsto  a_k^{\dagger}  \mapsto   \widehat{x_k}\,,\\
\D{p_k}  &\mapsto  k\; a_k            \mapsto  k\;\frac{\partial \phantom{F}}{\partial x_k}\,,
\end{align*}
where the $a_k$ are the creation operators, and the $a_k^{\dagger}$ are the annihilation operator. It is well-known that in the bosonic creation and annihilation operator algebra, the monomials in normal order
\[
(a_1^{\dagger})^{m_1}  (a_2^{\dagger})^{m_2} \cdots a_1^{n_1} a_2^{n_2} \cdots 
\]
as well as  the monomials in antinormal order
\[
 a_1^{n_1} a_2^{n_2} \cdots  (a_1^{\dagger})^{m_1}  (a_2^{\dagger})^{m_2} \cdots
\]
are linearly independent. This shows that the operators $\U{p_\lambda} \D{p_{\mu}}$ are linearly independent, and so are the operators $ \D{p_{\mu}} \U{p_\lambda}$. From this one deduces that  finite expansions with respect to $(U,D)$ and  with respect to $(D,U)$ are unique. 
\end{remark}

We finish this section with an expansion of the operators $\KB{f}$ in terms of operators $U_g$ and $D_g$.

From (2) in Lemma \ref{coordinate_independent}, we get the following result.
\begin{proposition}\label{KBf in U D}
Let $f$ be a symmetric function. The operator $\KB{f}$ lies in the subalgebra of $\operatorname{End}(\Sym)$ generated by the operators $U_g$ and $D_g$ (for $g \in \Sym$).

More precisely,
\[
\KB{f} = \sum_{\lambda} \U{f[X-1] \ast s_{\lambda}} \D{\lambda}.
\]
\end{proposition}

\begin{proof}
The subalgebra of $\operatorname{End}(\Sym)$ generated by the operators $U_g$ and $D_g$, for $g \in \Sym$, is the linear span of the operators $U_{\alpha} D_{\beta}$, for $\alpha$ and $\beta$ partitions. After (2) in Lemma \ref{coordinate_independent}, and the calculations in the proof of Proposition \ref{finite expansions}, it is the set of operators $M$ such that $M(\ss[XT])$ is $\ss[XT]$ times an element of $\Sym(X) \otimes_{\QQ} \Sym(T)$. 

After \eqref{vertex},$\KB{f}(\ss[XT])=\ss[X] f[X-1] * \ss[XT]$, which is equal, after \eqref{product}, to $\ss[XT] f[XT-1]$. This proves Proposition \ref{KBf in U D}.

To get an explicit decomposition of $\KB{f}$, we decompose $f[XT-1]$ as an element of $\Sym(X) \otimes_{\QQ} \Sym(T)$. We start with
\[
f[XT-1] = f[X-1] \ast \sigma[XT].
\]
(The Kronecker product $\ast$ is relative to the symmetric functions in $X$). Thus,
\[
f[XT-1] 
=  f[X-1] \ast \sum_{\lambda} s_{\lambda}[X] s_{\lambda}[T] 
= \sum_{\lambda} f[X-1]\ast s_{\lambda}[X]  s_{\lambda}[T].
\]
Therefore, 
\[
\KB{f}(\ss[XT]) = \Phi_{U,D}\left( \sum_{\lambda}  f[A-1] \ast s_{\lambda}[A] s_{\lambda}[B]\right)
\]
The conclusion comes by applying the second part of (2) in Lemma \ref{coordinate_independent}.
\end{proof}
For instance, for $f=h_{k}$ we have:
$
\KB{(k)} = \sum_{\lambda \vdash k} \U{\lambda} \D{\lambda} - \sum_{\lambda \vdash k-1} \U{\lambda} \D{\lambda}
,
$
since $h_k[X-1]=h_k-h_{k-1}$. The simplest case is $k=1$. Here we obtain
\begin{equation}
\KB{(1)}=\U{(1)}\D{(1)}-1\label{KB1}\,.
\end{equation}
By the Pieri rule, for any partition $\alpha$, 
\[
\U{(1)}\D{(1)} s_{\alpha} = \sum_{\beta} s_{\beta},
\]
where each term in the sum corresponds to a choice of a corner in the diagram of $\alpha$, that is removed, and then a choice of a box that is added, to give the diagram of a partition $\beta$. There are two cases: the box can be added where the corner was removed, or not. Accordingly the sum splits:
\[
\U{(1)}\D{(1)} s_{\alpha} = \noc{\alpha} s_{\alpha} +  \sum_{\beta \in \addremove{\alpha}} s_{\beta}.
\]
Therefore
\[
\KB{(1)} s_{\alpha}=\U{(1)}\D{(1)} s_{\alpha}-s_{\alpha} = (\noc{\alpha}-1) s_{\alpha} +  \sum_{\beta \in \addremove{\alpha}} s_{\beta}\,
\]
which is \eqref{equ:straightcorners}.


\section{Application to the skew Littlewood--Richardson rule} \label{SkewLRrule}

In this section we present our first application of Theorem~\ref{thm:main}: a new proof of the skew Littlewood--Richardson rule as conjectured by Assaf and the second author \cite{AssafMcNamara} and proved by Lam, Lauve and Sottile \cite{LamLauveSottile}.  As in \cite{LamLauveSottile}, our starting point is  \eqref{eqUalphaDbeta}. In \cite{LamLauveSottile}, first an ``algebraic skew Littlewood--Richardson rule'' is derived, involving sums of products of Littlewood--Richardson coefficients. Then, the combinatorial skew Littlewood--Richardson rule is obtained by interpreting these Littlewood--Richardson coefficients as counting semistandard Young tableaux with given rectification. Our proof fits more closely to the statement of the skew Littlewood--Richardson rule: we avoid going through the algebraic skew Littlewood--Richardson rule and use the interpretation of the Littlewood--Richardson coefficients as counting semistandard Young tableaux with content and Yamanouchi constraints.
Our proof appears in Subsections~\ref{sub:combskewLRrule} and \ref{sub:combskewLRrule2} and is largely combinatorial.  

%

For a positive integer $k$ and a partition $\gamma$, the classical Pieri rule \cite{Pieri} gives a simple and beautiful expression for the product $s_{(k)} s_\gamma$ as a sum of Schur functions. A \emph{$k$-horizontal} (resp.\ \emph{$k$-vertical}) \emph{strip} is a skew shape with $k$ boxes that has at most one box in each column (resp.\ row).  The Pieri rule states that
\[
s_{(k)} s_\gamma = \sum_\lambda s_{\plus{\gamma}}\,,
\]
where the sum is over all partitions $\plus{\gamma}$ such that $\plus{\gamma}/\gamma$ is a $k$-horizontal strip. In \cite{AssafMcNamara}, Assaf and the second author generalized the Pieri rule to the setting of skew shapes as follows:
\begin{equation}\label{equ:skewpieri}
s_{(k)} s_{\gamma/\beta} = \sum_{i=0}^k (-1)^i \sum_{\plus{\gamma},\minus{\beta}} s_{\plus{\gamma}/\minus{\beta}}\ ,
\end{equation}
where the sum is over all partitions $\plus{\gamma}$ and $\minus{\beta}$ such that $\plus{\gamma}/\gamma$ is a $(k-i)$-horizontal strip, and $\beta/\minus{\beta}$ is an $i$-vertical strip.  We will use the skew Pieri rule with $k=1$ in Section~\ref{SkewKronecker}.

For the next level of generality, it is natural to ask for a similarly combinatorial expression for $s_\alpha s_{\beta/\gamma}$ for any partition $\alpha$. Equation~\eqref{ReverseFoulkes} gives one expression, but it does not mimic \eqref{equ:skewpieri} in the sense that it does not give the answer as a signed sum of skew Schur functions.  Instead, the skew Littlewood--Richardson rule \cite{LamLauveSottile} gives an expression for the even more general product $s_{\alpha/\delta} s_{\beta/\gamma}$ as a signed sum of skew Schur functions.  In this section we will derive the skew Littlewood--Richardson rule from \eqref{eqUalphaDbeta} in the following way.  In Subsection~\ref{sub:combskewLRrule}, we will use a combinatorial approach to obtain from \eqref{eqUalphaDbeta} the skew Littlewood--Richardson rule in the case when $\delta$ is empty, and then we will use a linearity argument to derive the result for general $\delta$ in Subsection~\ref{sub:combskewLRrule2}.  

\subsection{The combinatorial skew Littlewood--Richardson rule}\label{sub:skewLRrule}

In order to state the skew Littlewood--Richardson rule, we first need some terminology.  As usual, a sequence of positive integers $\omega$ is said to be a \emph{lattice permutation} if any prefix of $\omega$ contains at least as many appearances of $i$ as $i+1$, for all $i\geq1$.  For a partition $\delta$, we will say that $\omega$ is a $\delta$-lattice permutation if the word obtained by prefixing $\omega$ with $\delta_1$ copies of 1 followed by $\delta_2$ copies of 2, etc., is a lattice permutation.  

We will draw our Young tableaux in French notation, implying that the entries of an SSYT weakly increase along the rows and strictly increase up the columns.  An \emph{anti-semistandard Young tableau} (ASSYT) $T_1$ of shape $\alpha/\beta$ is a filling of the boxes of $\alpha/\beta$ so that the entries strictly decrease along the rows and weakly decrease up the columns.  Equivalently, $T_1$ is an ASSYT if the tableau $(T_1')^r$ obtained by transposing $T_1$ and then rotating it 180$^\circ$ is an SSYT. The \emph{reverse reading word of an SSYT} $T_2$ is defined as usual as the word obtained by reading right-to-left along the rows of $T_2$, taking the rows from bottom to top.  In contrast, the \emph{reverse reading word of an ASSYT} $T_1$ is the word obtained by reading up the columns of $T$, taking the columns from right-to-left.  Equivalently, we can take the usual reverse reading work of the SSYT $(T_1')^r$.
Given a pair of tableaux $(T_1, T_2)$, where $T_1$ is an ASSYT and $T_2$ is an SSYT, we define the \emph{reverse reading word of the pair} as the concatenation of the reading word of $T_1$ with that of $T_2$.  We will encounter such pairs as in the the figure below, where the entries in the bottom left form an ASSYT, and the entries above or to the right of the outlined skew shape form an SSYT.  
\begin{equation}\label{equ:ASSYTSSYT}
\begin{tikzpicture}[scale=0.5]
\draw[thick] (3,0) -- (3,2) -- (0,2) -- (0,4) -- (1,4) -- (1,3) -- (4,3) -- (4,2) -- (5,2) -- (5,1) -- (7,1) -- (7,0) --cycle;
\draw (2.5, 0.5) node {2};
\draw (2.5, 1.5) node {1};
\draw (1.5, 0.5) node {3};
\draw (1.5, 1.5) node {3};
\draw (0.5, 1.5) node {5};
\draw (7.5, 0.5) node {2};
\draw (8.5, 0.5) node {4};
\draw (5.5, 1.5) node {1};
\draw (6.5, 1.5) node {4};
\draw (7.5, 1.5) node {4};
\draw (8.5, 1.5) node {5};
\draw (4.5, 2.5) node {3};
\draw (1.5, 3.5) node {5};
\draw (2.5, 3.5) node {6};
\end{tikzpicture}
\end{equation}
The reverse reading word of $(T_1, T_2)$ shown in \eqref{equ:ASSYTSSYT} is 21335425441365, which is certainly not a lattice permutation but is a 5321-lattice permutation.  

We are now ready to state the skew Littlewood--Richardson rule.

\begin{theorem}[Conjecture 6.1 of \cite{AssafMcNamara}; Theorem 3.2 and Remark 3.3(ii) of \cite{LamLauveSottile}]\label{thm:skewLR}
For skew shapes $\alpha/\delta$ and $\gamma/\beta$,
\begin{equation}\label{eqskewLR}
s_{\alpha/\delta} s_{\gamma/\beta} = \sum_{\genfrac{}{}{0pt}{}{T_1 \in \mathrm{ASSYT}(\beta/\minus{\beta})}{T_2 \in \mathrm{SSYT}(\plus{\gamma}/\gamma)}} (-1)^{|\beta/\minus{\beta}|} s_{\plus{\gamma}/\minus{\beta}}\ ,
\end{equation}
where the sum is over all ASSYT $T_1$ of shape $\beta/\minus{\beta}$ for some
$\minus{\beta} \subseteq \beta$, and SSYT $T_2$ of shape $\plus{\gamma}/\gamma$ for some
$\plus{\gamma} \supseteq \gamma$, with the following properties:

\begin{enumerate}
\renewcommand{\theenumi}{\alph{enumi}}
\item the combined content of $T_1$ and $T_2$ is the component-wise difference $\alpha-\delta$, and
\item the reverse reading word of $(T_1,T_2)$ is a $\delta$-lattice permutation.
\end{enumerate}
\end{theorem}

For example, the ASSYT and SSYT pair of \eqref{equ:ASSYTSSYT} contribute $-s_{9953/1}$ to the product $s_{755431/5321} s_{7541/33}$.  Note that when $\beta$ and $\delta$ are empty, we recover the classical Littlewood--Richardson rule.

\subsection{Recovering a special case of the combinatorial skew Littlewood--Richardson rule}
\label{sub:combskewLRrule}
Our first step to reproving Theorem~\ref{thm:skewLR} is to start with \eqref{eqUalphaDbeta} and show it implies Theorem~\ref{thm:skewLR} in the case when $\delta=(0)$, the empty partition.  Instead of \eqref{eqUalphaDbeta}, we work with the equivalent \eqref{ReverseFoulkes}: 

\[
s_{\alpha} s_{\gamma/\beta}=\sum_{\lambda} (-1)^{|\lambda|} \D{\beta/\lambda'} (s_{\alpha/\lambda}s_{\gamma}).
\]

First, let us examine the product $s_{\alpha/\lambda}s_{\gamma}$ from the right--hand side, and expand it in terms of Schur functions.  Note that only those $s_\nu$ with $\nu \supseteq \gamma$ will appear in the Schur expansion with nonzero coefficient.  Thus we can write 
\[
s_{\alpha/\lambda}s_{\gamma} = \sum_{\plus{\gamma} \supseteq \gamma} a_{\plus{\gamma}} s_{\plus{\gamma}}\,.
\]
We have 
\[
a_{\plus{\gamma}} = \scalar{s_{\plus{\gamma}}}{s_{\alpha/\lambda}s_{\gamma}} = \scalar{s_{\plus{\gamma}/\gamma}}{s_{\alpha/\lambda}} = \scalar{s_{\plus{\gamma}/\gamma} s_\lambda}{s_\alpha}.
\]
The product $s_{\plus{\gamma}/\gamma} s_\lambda$ is equal to the skew Schur function of the shape $(\plus{\gamma}/\gamma) \oplus \lambda$.  (The notation $\oplus$ denotes that $\plus{\gamma}/\gamma$ is positioned so that its bottom-right corner box is immediately northwest of the top-left corner box of $\lambda$.)  Therefore, the coefficient $a_{\plus{\gamma}}$ is equal to the number of Littlewood--Richardson fillings (LR-fillings) of that skew shape that have content $\alpha$.  Any LR-filling of that shape must just fill the $i$th row of $\lambda$ with the number $i$, for all $i$.  Thus $a_{\plus{\gamma}}$ equals the number of SSYT of shape $\plus{\gamma}/\gamma$ whose reverse reading word is a  $\lambda$-lattice permutation and whose content is the component-wise difference $\alpha-\lambda$.  Hence \eqref{ReverseFoulkes} is equivalent to
\begin{equation}\label{equ:gammayamanouchi}
s_{\alpha} s_{\gamma/\beta} = \sum_{\lambda} (-1)^{|\lambda|} \D{\beta/\lambda'} \sum_{T_2} s_{\plus{\gamma}}
\end{equation}
where second sum is over all SSYT $T_2$ of shape $\plus{\gamma}/\gamma$ for some $\plus{\gamma}\supseteq \gamma$, and content $\alpha-\lambda$, whose reverse reading word is a $\lambda$-lattice permutation.

Next, we will examine the term $s_{\beta/\lambda'}$.  The coefficient of $s_\nu$ in this term is exactly the Littlewood--Richardson coefficient $c^\beta_{\lambda'\nu}$\,, which is nonzero only if $\nu \subseteq \beta$.  Thus we wish to determine the coefficient of $s_{\minus{\beta}}$ in $s_{\beta/\lambda'}$ when $\minus{\beta} \subseteq \beta$, which equals the number of LR-fillings of $\beta/\minus{\beta}$ of content $\lambda'$. We claim that such fillings $T$ are in a shape-preserving bijection with ASSYT that are lattice permutations of content $\lambda$ (as opposed to content $\lambda'$ previously).  Indeed the bijection $\psi$ is defined as mapping the $i$th appearance (in the reverse reading word of the SSYT $T$) of the number $j$ to the number $i$, for all $i$ and $j$.  For example,
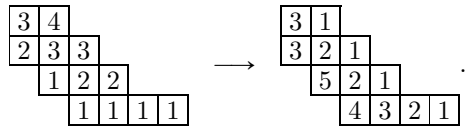
\[
\Yvcentermath1
\young(34,233,:122,::1111)\ \  \longrightarrow\ \  \young(31,321,:521,::4321).
\]
Then, one can check that $\psi$ has the following necessary properties.
\begin{itemize}
\item The inverse of $\psi$ is given by the ASSYT analogue of $\psi$: map the $j$th appearance (in the reverse reading word, now in the ASSYT sense) of the number $i$ to the number $j$, for all $i$ and $j$.
\item The image $\psi(T)$ of an LR-filling $T$ is indeed an ASSYT whose reverse reading word is a lattice permutation.
\item Such a $\psi(T)$ maps to an LR-filling under the inverse map.  
\item Both $\psi$ and its inverse transpose the content partition.
\end{itemize}

Thus \eqref{equ:gammayamanouchi} is equivalent to
\[
s_{\alpha} s_{\gamma/\beta} = \sum_{\lambda} (-1)^{|\lambda|} \sum_{T_1} \D{\minus{\beta}}
\sum_{T_2} s_{\plus{\gamma}} = \sum_{\lambda} (-1)^{|\lambda|} \sum_{T_1, T_2} s_{\plus{\gamma}/\minus{\beta}}\,,
\]
where the relevant sums are over all $T_1$ and $T_2$ such that
\begin{itemize}
\item $T_1$ is an ASSYT having content $\lambda$, whose reverse reading word is a lattice permutation, and with shape $\beta/\minus{\beta}$ for some $\minus{\beta} \subseteq \beta$, and
\item $T_2$ is an SSYT having content $\alpha-\lambda$, a $\lambda$-lattice permutation as reverse reading word, and shape $\plus{\gamma}/\gamma$ for some $\plus{\gamma}\supseteq\gamma$.
\end{itemize}
Note that $T_1$ tells us that $|\lambda| = |\beta/\minus{\beta}|$, and we have arrived at Theorem~\ref{thm:skewLR} in the case when $\delta=(0)$.

\subsection{Recovering the full combinatorial skew Littlewood--Richardson rule}\label{sub:combskewLRrule2}

Our second step is to use a linearity argument to derive Theorem~\ref{thm:skewLR} for general $\delta$.  For this, observe that the coefficient of $(-1)^{|\beta/\minus{\beta}|} s_{\plus{\gamma}/\minus{\beta}}$ on the right-hand side of \eqref{eqskewLR} is the number of pairs of tableaux $(T_1, T_2)$ with $T_1$ an ASSYT and $T_2$ a SSYT, fulfilling conditions (a) and (b) in Theorem~\ref{thm:skewLR}. But $T_1$ being an ASSYT is equivalent to $(T_1')^r$ being an SSYT, and the reverse reading word of the ASSYT $T_1$ is defined so that $(T_1')^r$ has the same reverse reading word as an SSYT. Therefore, the coefficient of $(-1)^{|\beta/\minus{\beta}|} s_{\plus{\gamma}/\minus{\beta}}$ on the right-hand side of \eqref{eqskewLR} equals the number of SSYT of shape $(\plus{\gamma}/\gamma) \oplus (\beta'/\minus{\beta'})^r$ and content $\alpha$ whose reverse reading word is a $\delta$--lattice permutation.  This is the number of SSYT of shape $(\plus{\gamma}/\gamma) \oplus (\beta'/\minus{\beta'})^r \oplus \delta$ and content $\alpha$ whose reverse reading word is a lattice  permutation.  By the Littlewood--Richardson rule, this quantity equals the coefficient of $s_\alpha$ in the Schur expansion of $s_{(\plus{\gamma}/\gamma) \oplus (\beta'/\minus{\beta'})^r \oplus \delta}$. This skew Schur function being equal to the product  $s_{\plus{\gamma}/\gamma} s_{(\beta'/\minus{\beta'})^r}s_{\delta}$, this coefficient is equal to
\[
\scalar{s_{\plus{\gamma}/\gamma} s_{(\beta'/\minus{\beta'})^r}s_\delta}{s_\alpha}\,,
\]
which is equal to
\[
\scalar{s_{\plus{\gamma}/\gamma} s_{(\beta'/\minus{\beta'})^r}}{s_{\alpha/\delta}}\,.
\]
Therefore, \eqref{eqskewLR} is equivalent to 
\begin{equation}\label{equ:altskewLR}
s_{\alpha/\delta}  s_{\gamma/\beta} =  \sum_{\plus{\gamma}, \minus{\beta'}} (-1)^{|\beta/\minus{\beta}|} \scalar{s_{\plus{\gamma}/\gamma} s_{(\beta'/\minus{\beta'})^r}}{s_{\alpha/\delta}} s_{\plus{\gamma}/\minus{\beta}}\,,
\end{equation}
where the sums are over all partitions $\plus{\gamma}$ and $\minus{\beta'}$ such that $\plus{\gamma} \supseteq \gamma$ and $\minus{\beta} \subseteq \beta$. 

The key observation is that \eqref{equ:altskewLR} is linear in $s_{\alpha/\delta}$. Given any partitions $\beta$ and $\gamma$, \eqref{equ:altskewLR} will be true for all partitions $\alpha$ and $\delta$ when 
\[
f \cdot s_{\gamma/\beta} = \sum_{\plus{\gamma}, \minus{\beta'}} (-1)^{|\beta/\minus{\beta}|} \scalar{s_{\plus{\gamma}/\gamma} s_{(\beta'/\minus{\beta'})^r}}{f} s_{\plus{\gamma}/\minus{\beta}}\ ,
\]
holds for any symmetric function $f$. Since the Schur functions form a basis for the space of symmetric functions, it is enough to check it for $f=s_\alpha$, all partitions $\alpha$. This is what was done in Subsection~\ref{sub:combskewLRrule}.

\section{A combinatorial interpretation for the Kronecker product of a skew Schur function by  $s_{(n-1,1)}$.}\label{SkewKronecker}

As another application of the identities of Section~\ref{sec:commutations}, our goal for this section is to derive a combinatorial formula for Kronecker products involving skew Schur functions. Let $\alpha$ be a partition of $n$, and let us speak of partitions and their Young diagrams interchangeably.   

A well-known case of  $\KB{\lambda}$ is when $\lambda=(1)$. We pause to describe an identity that we will generalize in Section \ref{SkewKronecker} using the normal ordering relations. 
A \emph{corner} of $\alpha$ is a box of $\alpha$ whose removal results in another partition, and we denote by $\noc{\alpha}$ the number of corners of $\alpha$.  Denote by $\remove{\alpha}$ the set of partitions that result from removing a corner of $\alpha$.  Similarly, $\add{\alpha}$ will denote the set of those partitions $\beta$ such that $\alpha \in \remove{\beta}$.  We use $\addremove{\alpha}$ to denote the set of partitions not equal to $\alpha$ that can be obtained by removing a corner of $\alpha$ and then adding a box to the result.  Equivalently, $\addremove{\alpha}$ is the set of partitions that can be obtained from $\alpha$ by first adding a box and then removing a different box.   For example, $31$ has two corners, and $\addremove{(31)} = \{(4), (22), (211)\}$.
We finish Section~\ref{uniqueness} by relating the combinatorial identity
\begin{equation}\label{equ:straightcorners}
\KB{(1)}s_{\alpha} = (\noc{\alpha} - 1) s_\alpha + \sum_{\beta \in \addremove{\alpha}} s_\beta\,.
\end{equation}
to the decomposition of the operators $\KB{f}$ in terms of operator $U_g$ and $D_g$.

We aim to generalize \eqref{equ:straightcorners} to skew Schur functions. This leads to our next use of the identities of Section 2.
Corollary~\ref{commutators} implies the relation $[\D{\theta}, \KB{(1)}]=\D{s_{\theta/(1)} s_1}$, 
which gives
\[
\KB{(1)} \D{\theta} (s_\alpha) =  \D{\theta} \KB{(1)} (s_\alpha) - \D{s_{\theta/(1)} s_1} (s_\alpha). 
\]
Applying \eqref{equ:straightcorners} and the fact that \[s_{\theta/(1)} s_1 = \noc{\theta}s_\theta + \sum_{\phi \in \addremove{\theta}} s_\phi,\]
we get
\[
s_{\alpha/\theta} * s_{(n-|\theta|-1,1)} 
= (\noc{\alpha} - \noc{\theta} - 1)s_{\alpha/\theta}  + \sum_{\beta \in \addremove{\alpha}} s_{\beta/\theta} - \sum_{\phi \in \addremove{\theta}} s_{\alpha/\phi}\ ,
\]

Thus we have an algebraic proof of the following.

\begin{theorem}\label{thm:skewcorners}
Suppose $\alpha \vdash n$ and $\theta \vdash k$ with $\theta \subseteq \alpha$.  Then
\[
s_{\alpha/\theta} * s_{(n-k-1,1)} = (\noc{\alpha} - \noc{\theta} - 1)s_{\alpha/\theta} + \sum_{\beta \in \addremove{\alpha}} s_{\beta/\theta} - \sum_{\phi \in \addremove{\theta}} s_{\alpha/\phi}\, .
\]
\end{theorem}

This results begs for a  a combinatorial proof.  We offer a proof which is ``two-thirds'' combinatorial.  The part which is non-combinatorial makes use of \eqref{KB1}, which is turn is proved using Littlewood's Identity \eqref{TrueLittlewood}.

\begin{proof}[Proof 2 of Theorem \ref{thm:skewcorners}]
The proof will work in three stages.  In the short first stage, which is the non-combinatorial one, we will apply \eqref{KB1} to express $s_{\alpha/\theta} * s_{(n-k-1,1)}$ in a form \eqref{equ:nokronecker} not involving any Kronecker products.  Then, using the skew Pieri rule \eqref{equ:skewpieri}, we will reduce the problem to showing an identity \eqref{equ:tableauxmanipulation} that is effectively purely about SSYT.  This identity will be proved in the third stage using jeu de taquin.

First, we will need to determine the result of applying $\D{(1)}$ to the skew Schur function $s_{\alpha/\theta}$.  We have
\[
\scalar{s_\beta}{\D{(1)} s_{\alpha/\theta}} 
= \scalar{s_\beta s_{(1)}}{s_{\alpha/\theta}} 
= \sum_{\delta\in\add{\theta}} \scalar{s_\beta s_\delta}{s_\alpha} 
=   \scalar{s_\beta}{\sum_{\delta\in\add{\theta}} s_{\alpha/\delta}},
\]
and so $\D{(1)} s_{\alpha/\theta} = \sum_{\delta\in\add{\theta}} s_{\alpha/\delta}$.
Applying \eqref{KB1} to $s_{\alpha/\theta}$, we immediately deduce
\begin{equation}\label{equ:nokronecker}
s_{\alpha/\theta} * s_{(n-k-1,1)} = s_{(1)} \sum_{\delta \in \add{\theta}} s_{\alpha/\delta} - s_{\alpha/\theta}\,.
\end{equation}

We next wish to apply the skew Pieri rule to $s_{(1)} \sum_{\delta \in \add{\theta}} s_{\alpha/\delta}$, but it will prove worthwhile to perform a preliminary step.  By definition, $s_{\alpha/\delta} = 0$ unless $\delta \subseteq \alpha$, so it suffices to sum over those $\delta$ such that $\delta \in \add{\theta}$ and $\delta \subseteq \alpha$.  We will denote that $\delta$ satisfies both conditions by writing $\delta \in\addrestrict{\theta}{\alpha}$.  So we now apply the skew Pieri rule to 
\[
s_{\alpha/\theta} * s_{(n-k-1,1)} = s_{(1)} \sum_{\delta \in \addrestrict{\theta}{\alpha}} s_{\alpha/\delta} - s_{\alpha/\theta}
\]
to yield
\[
s_{\alpha/\theta} * s_{(n-k-1,1)} 
= \sum_{\genfrac{}{}{0pt}{}{\gamma \in \add{\alpha}}{\delta \in \addrestrict{\theta}{\alpha}}} s_{\gamma/\delta} 
- \sum_{\genfrac{}{}{0pt}{}{\delta \in \addrestrict{\theta}{\alpha}}{\phi\in \remove{\delta}}} s_{\alpha/\phi}
- s_{\alpha/\theta}\ .
\]

Let us examine the second sum.  For any $\delta$ in $\addrestrict{\theta}{\alpha}$, we can choose $\phi = \theta$.  We see that the other $\phi$ that arise will be exactly those elements of $\addremove{\theta}$ that are contained in $\alpha$.  Therefore, 
\[
s_{\alpha/\theta} * s_{(n-k-1,1)} = \sum_{\genfrac{}{}{0pt}{}{\gamma \in \add{\alpha}}{\delta \in \addrestrict{\theta}{\alpha}}} s_{\gamma/\delta} - |\addrestrict{\theta}{\alpha}| s_{\alpha/\theta} -  \sum_{\phi \in \addremove{\theta}} s_{\alpha/\phi} - s_{\alpha/\theta}\ .
\]
Thus to prove Theorem~\ref{thm:skewcorners}, it remains to show that
\begin{equation}\label{equ:tableauxmanipulation}
\sum_{\genfrac{}{}{0pt}{}{\gamma \in \add{\alpha}}{\delta \in \addrestrict{\theta}{\alpha}}} s_{\gamma/\delta} - |\addrestrict{\theta}{\alpha}| s_{\alpha/\theta}
=  (\noc{\alpha} - \noc{\theta})s_{\alpha/\theta} + \sum_{\beta \in \addremove{\alpha}} s_{\beta/\theta}\ .
\end{equation}

Our main tool for proving the above identity will be jeu de taquin but, like with our application of the skew Pieri rule, it will be worthwhile to rewrite \eqref{equ:tableauxmanipulation} in a slightly different form.  Observe that for any partition $\alpha$, we have $\noc{\alpha} = |\add{\alpha}|-1$.  
For $\theta \subseteq \alpha$, denote those elements of $\add{\theta}$ that are not contained in $\alpha$ by $\addcomplement{\theta}{\alpha}$, which we can check can only be non-empty if $\alpha/\theta$ has some empty rows or columns.
  We can now rewrite \eqref{equ:tableauxmanipulation} as
\begin{equation}\label{equ:tableauxmanipulation2}
\sum_{\genfrac{}{}{0pt}{}{\gamma \in \add{\alpha}}{\delta \in \addrestrict{\theta}{\alpha}}} s_{\gamma/\delta} 
=  (|\add{\alpha}| - |\addcomplement{\theta}{\alpha}|)s_{\alpha/\theta} + \sum_{\beta \in \addremove{\alpha}} s_{\beta/\theta}\ .
\end{equation}
For intuition, we can call the positions of the form $\lambda/\alpha$ for some $\lambda \in \add{\alpha}$ the \emph{outside corners} of $\alpha$.  Then the term $|\add{\alpha}| - |\addcomplement{\theta}{\alpha}|$ is the number of outside corners of $\alpha$, excluding those that are also outside corners of $\theta$.  See Example~\ref{exa:jdt} below for a fully worked example of the remainder of the proof.

To prove \eqref{equ:tableauxmanipulation2} using jeu de taquin (jdt), consider an SSYT $T$ that contributes to the left-hand side, meaning $T$ has shape $\gamma/\delta$, where $\gamma \in \add{\alpha}$ and $\delta \in \addrestrict{\theta}{\alpha}$.  Notice that the unique box $b$ of $\delta/\theta$ is not an element of $T$, and that the unique box $c$ of $\gamma/\alpha$ is in $T$. Perform a jdt slide of $T$ into $b$, and let $T'$ denote the resulting SSYT.  There are three possibilities that can arise.
\begin{enumerate}
\renewcommand{\theenumi}{\alph{enumi}}
\item $T'=T$, meaning there is no way to fill $b$ under a jdt slide of $T$. 
\item $T'$ contains $b$, and the single vacated box under the jdt slide is not $c$.
\item $T'$ contains $b$, and the single vacated box under the jdt slide is $c$.
\end{enumerate}

By definition of jdt, Case (a) can happen if and only if $b$ is a corner of $\gamma$.  Since $b \in \delta \subseteq \alpha \subset \gamma$, it must also be the case that $b$ is a corner of $\alpha$.  Therefore, since $b$ is the unique box of $\delta/\theta$ and is not an element of $T'$ while $c$ is the unique box of $\gamma/\alpha$ and is in $T'$, the shape $\gamma/\delta$ of $T'$ can be written in the equivalent form $\beta/\theta$, where $\beta$ is obtained from $\alpha$ by removing $b$ and adding $c$.  In particular, $\beta \in \addremove{\alpha}$, and so such $T'$ contribute part of the sum on the right of \eqref{equ:tableauxmanipulation2}.

We claim that Case (b) contributes the rest of the sum on the right of \eqref{equ:tableauxmanipulation2}.  
We see that the shape of $T'$ is obtained from the shape of $T$ by making exactly two changes: $T'$ contains $b$, and a box $c'$ different from $c$ has been vacated.  As a result, $T'$ has shape $\beta/\theta$ for some $\beta \in \addremove{\alpha}$. 

To prove our claim from the start of the previous paragraph, let $T'$ be an SSYT of shape $\beta/\theta$, where $\beta \in \addremove{\alpha}$.  We wish to show that $T'$ arises as the image under jdt of exactly one $T$ from Cases (a) and (b).  In short, the reason is that jdt slides are reversible, but let us be more precise.  Suppose $\beta$ is obtained from $\alpha$ by removing a box $d$ and adding a different box $e$.  Perform a jdt slide of $T'$ into $d$. (Some references would call this a \emph{reverse} jdt slide, since $d$ is outside $\beta$.)  There are two possibilities that can arise.
\begin{enumerate}
\renewcommand{\theenumi}{\roman{enumi}}
\item There is no way to fill $d$ under a jdt slide of $T'$.  This can happen if and only if $d$ is an outside corner of $\theta$.  Thus $\beta/\theta$ can be written in the equivalent form $\gamma/\delta$, where $\gamma$ equals $\alpha$ with $e$ added, and $\delta$ equals $\theta$ with $d$ added.  Since $d$ is a corner of $\alpha$, we also have that $d$ is a corner of $\gamma$.  These are exactly the conditions for $T'$ to arise as an image under jdt of a $T$ from Case (a) (where in fact $T=T'$ and our $d$ here corresponds to $b$ in Case (a)).
\item Performing a jdt slide of $T'$ into $d$ fills $d$ and vacates an outside corner of $\theta$.  This is exactly the reverse of the jdt slide from Case (b) (where $d$ is playing the role of $c'$). \end{enumerate}
Thus $T'$ arises as the image of a single $T$ from Cases (a) and (b).  

It remains to consider Case (c).  Note that all $T'$ in Case (c) are of shape $\alpha/\theta$.  We would like to show that each $T'$ is the image under jdt of $k$ distinct $T$, where $k=(|\add{\alpha}| - |\addcomplement{\theta}{\alpha}|)$.  This would show that the $T'$ from Case (c) together contribute the term
$(|\add{\alpha}| - |\addcomplement{\theta}{\alpha}|)s_{\alpha/\theta}$ from the right-hand side of \eqref{equ:tableauxmanipulation2}, and \eqref{equ:tableauxmanipulation2} would be proved.  

So pick a $T'$ from Case (c).  Pick an outside corner $c$ of $\alpha$, and perform a (reverse) jdt slide of $T'$ into $c$.  If $c$ is not an outside corner of $\theta$, then this jdt slide will fill $c$, and the result will be an SSYT $T$ of shape $\gamma/\delta$ with $\gamma \in \add{\alpha}$ and $\delta \in \addrestrict{\theta}{\alpha}$.  These are exactly the conditions for $T'$ to arise in Case (c).  

On the other hand, if $c$ is an outside corner of $\alpha$ and also of $\theta$, then $T'$ will remain fixed under the (reverse) jdt slide.  Thus any $T$ that maps to such a $T'$ under a jdt slide must also have shape $\alpha/\theta$.  Such a $T$ from the left-hand side of \eqref{equ:tableauxmanipulation2} does not exist, since $T$ would contain the single box $\gamma/\alpha$ whereas $T'$ does not.  
We conclude that each $T'$ in Case (c) is the image of exactly $k$ distinct $T$, as required.
\end{proof}

\begin{example}\label{exa:jdt}
Suppose $\alpha=(4,1,1)$ and $\theta=(2,1)$.
\[
\begin{tikzpicture}[scale=0.4]
\tikzstyle{every node}=[inner sep=3pt]; 
\draw (-1.5,1.5) node {$\alpha/\theta=$};
\begin{scope}[xshift=2mm]
\draw[thick] (0,2) rectangle (1,3);
\draw[thick] (2,0) rectangle (3,1);
\draw[thick] (3,0) rectangle (4,1);
\draw[dashed] (0,0) rectangle (1,1);
\draw[dashed] (0,1) rectangle (1,2);
\draw[dashed] (1,0) rectangle (2,1);
\end{scope}
\end{tikzpicture}
\]
The complete set of shapes $\gamma/\delta$, and SSYT $T$ and $T'$ from the proof of \eqref{equ:tableauxmanipulation2} are show in Table~\ref{tab:jdt}.  Deliberately omitted from the table is the scenario from the last paragraph of the proof, where $c$ is an outside corner of both $\alpha$ and $\theta$, which does not contribute to either side of \eqref{equ:tableauxmanipulation2}.  In this example, there is $1=|\addcomplement{\theta}{\alpha}|$ such situation, shown below.
\[
\begin{tikzpicture}[scale=0.4]
\tikzstyle{every node}=[inner sep=3pt]; 
\draw (-1.5,1.5) node {$\alpha/\theta=$};
\begin{scope}[xshift=2mm]
\draw[thick] (0,2) rectangle (1,3);
\draw[thick] (2,0) rectangle (3,1);
\draw[thick] (3,0) rectangle (4,1);
\draw[dashed] (0,0) rectangle (1,1);
\draw[dashed] (0,1) rectangle (1,2);
\draw[dashed] (1,0) rectangle (2,1);
\draw (1.5,1.5) node {$c$};
\end{scope}
\end{tikzpicture}
\]

\begin{table}
\begin{tabular}{c||c|c||c|c||c|c||}
& \multicolumn{2}{c||}{Case (a)} & \multicolumn{2}{c||}{Case (b)} & \multicolumn{2}{c||}{Case (c)} 
 \\ \hline
\begin{tikzpicture}[scale=0.32]
\draw (0,1) node {$\gamma/\delta$};
\draw (0,0) node {};
\draw (0,4) node {};   
\end{tikzpicture}
& 
\begin{tikzpicture}[scale=0.32]
\begin{scope}
\draw[thick] (1,1) rectangle (2,2);
\draw[thick] (2,0) rectangle (3,1);
\draw[thick] (3,0) rectangle (4,1);
\draw[dashed] (0,0) rectangle (1,1);
\draw[dashed] (0,1) rectangle (1,2);
\draw[dashed] (1,0) rectangle (2,1);
\draw[dashed] (0,2) rectangle (1,3);
\draw (0.5,2.5) node {\footnotesize $b$};
\draw (1.5,1.5) node {\footnotesize $c$};
\end{scope}
\end{tikzpicture}
& 
\begin{tikzpicture}[scale=0.32]
\begin{scope}
\draw[thick] (4,0) rectangle (5,1);
\draw[thick] (2,0) rectangle (3,1);
\draw[thick] (3,0) rectangle (4,1);
\draw[dashed] (0,0) rectangle (1,1);
\draw[dashed] (0,1) rectangle (1,2);
\draw[dashed] (1,0) rectangle (2,1);
\draw[dashed] (0,2) rectangle (1,3);
\draw (0.5,2.5) node {\footnotesize $b$};
\draw (4.5,0.5) node {\footnotesize $c$};
\end{scope}
\end{tikzpicture}
& 
\begin{tikzpicture}[scale=0.32]
\begin{scope}
\draw[thick] (0,3) rectangle (1,4);
\draw[thick] (0,2) rectangle (1,3);
\draw[thick] (3,0) rectangle (4,1);
\draw[dashed] (0,0) rectangle (1,1);
\draw[dashed] (0,1) rectangle (1,2);
\draw[dashed] (1,0) rectangle (2,1);
\draw[dashed] (2,0) rectangle (3,1);
\draw (2.5,0.5) node {\footnotesize $b$};
\draw (0.5,3.5) node {\footnotesize $c$};
\draw (3.5,0.5) node {\footnotesize $c'$};
\end{scope}
\end{tikzpicture}
& 
\begin{tikzpicture}[scale=0.32]
\begin{scope}
\draw[thick] (3,0) rectangle (4,1);
\draw[thick] (0,2) rectangle (1,3);
\draw[thick] (1,1) rectangle (2,2);
\draw[dashed] (0,0) rectangle (1,1);
\draw[dashed] (0,1) rectangle (1,2);
\draw[dashed] (1,0) rectangle (2,1);
\draw[dashed] (2,0) rectangle (3,1);
\draw (2.5,0.5) node {\footnotesize $b$};
\draw (1.5,1.5) node {\footnotesize $c$};
\draw (3.5,0.5) node {\footnotesize $c'$};
\end{scope}
\end{tikzpicture}
& 
\begin{tikzpicture}[scale=0.32]
\begin{scope}
\draw[thick] (4,0) rectangle (5,1);
\draw[thick] (0,2) rectangle (1,3);
\draw[thick] (3,0) rectangle (4,1);
\draw[dashed] (0,0) rectangle (1,1);
\draw[dashed] (0,1) rectangle (1,2);
\draw[dashed] (1,0) rectangle (2,1);
\draw[dashed] (2,0) rectangle (3,1);
\draw (2.5,0.5) node {\footnotesize $b$};
\draw (4.5,0.5) node {\footnotesize $c$};
\end{scope}
\end{tikzpicture}
& 
\begin{tikzpicture}[scale=0.32]
\begin{scope}
\draw[thick] (0,3) rectangle (1,4);
\draw[thick] (2,0) rectangle (3,1);
\draw[thick] (3,0) rectangle (4,1);
\draw[dashed] (0,0) rectangle (1,1);
\draw[dashed] (0,1) rectangle (1,2);
\draw[dashed] (1,0) rectangle (2,1);
\draw[dashed] (0,2) rectangle (1,3);
\draw (0.5,2.5) node {\footnotesize $b$};
\draw (0.5,3.5) node {\footnotesize $c$};
\end{scope}
\end{tikzpicture} 
\\ [2ex]
\begin{tikzpicture}[scale=0.32]
\draw (0,1) node {$T$};
\draw (0,0) node {};
\end{tikzpicture}
&
\begin{tikzpicture}[scale=0.32]
\begin{scope}
\draw[thick] (1,1) rectangle (2,2);
\draw[thick] (2,0) rectangle (3,1);
\draw[thick] (3,0) rectangle (4,1);
\draw[dashed] (0,0) rectangle (1,1);
\draw[dashed] (0,1) rectangle (1,2);
\draw[dashed] (1,0) rectangle (2,1);
\draw[dashed] (0,2) rectangle (1,3);
\draw (0.5,2.5) node {\footnotesize $b$};
\draw (1.5,1.5) node {\footnotesize $A$};
\draw (2.5,0.5) node {\footnotesize $B$};
\draw (3.5,0.5) node {\footnotesize $C$};
\end{scope}
\end{tikzpicture}
& 
\begin{tikzpicture}[scale=0.32]
\begin{scope}
\draw[thick] (4,0) rectangle (5,1);
\draw[thick] (2,0) rectangle (3,1);
\draw[thick] (3,0) rectangle (4,1);
\draw[dashed] (0,0) rectangle (1,1);
\draw[dashed] (0,1) rectangle (1,2);
\draw[dashed] (1,0) rectangle (2,1);
\draw[dashed] (0,2) rectangle (1,3);
\draw (0.5,2.5) node {\footnotesize $b$};
\draw (2.5,0.5) node {\footnotesize \footnotesize $A$};
\draw (3.5,0.5) node {\footnotesize $B$};
\draw (4.5,0.5) node {\footnotesize $C$};
\end{scope}
\end{tikzpicture}
& 
\begin{tikzpicture}[scale=0.32]
\begin{scope}
\draw[thick] (0,3) rectangle (1,4);
\draw[thick] (0,2) rectangle (1,3);
\draw[thick] (3,0) rectangle (4,1);
\draw[dashed] (0,0) rectangle (1,1);
\draw[dashed] (0,1) rectangle (1,2);
\draw[dashed] (1,0) rectangle (2,1);
\draw[dashed] (2,0) rectangle (3,1);
\draw (2.5,0.5) node {\footnotesize $b$};
\draw (0.5,2.5) node {\footnotesize $A$};
\draw (0.5,3.5) node {\footnotesize $B$};
\draw (3.5,0.5) node {\footnotesize $C$};
\end{scope}
\end{tikzpicture}
& 
\begin{tikzpicture}[scale=0.32]
\begin{scope}
\draw[thick] (3,0) rectangle (4,1);
\draw[thick] (0,2) rectangle (1,3);
\draw[thick] (1,1) rectangle (2,2);
\draw[dashed] (0,0) rectangle (1,1);
\draw[dashed] (0,1) rectangle (1,2);
\draw[dashed] (1,0) rectangle (2,1);
\draw[dashed] (2,0) rectangle (3,1);
\draw (2.5,0.5) node {\footnotesize $b$};
\draw (0.5,2.5) node {\footnotesize $A$};
\draw (1.5,1.5) node {\footnotesize $B$};
\draw (3.5,0.5) node {\footnotesize $C$};
\end{scope}
\end{tikzpicture}
& 
\begin{tikzpicture}[scale=0.32]
\begin{scope}
\draw[thick] (4,0) rectangle (5,1);
\draw[thick] (0,2) rectangle (1,3);
\draw[thick] (3,0) rectangle (4,1);
\draw[dashed] (0,0) rectangle (1,1);
\draw[dashed] (0,1) rectangle (1,2);
\draw[dashed] (1,0) rectangle (2,1);
\draw[dashed] (2,0) rectangle (3,1);
\draw (2.5,0.5) node {\footnotesize $b$};
\draw (0.5,2.5) node {\footnotesize $A$};
\draw (3.5,0.5) node {\footnotesize $B$};
\draw (4.5,0.5) node {\footnotesize $C$};
\end{scope}
\end{tikzpicture}
& 
\begin{tikzpicture}[scale=0.32]
\begin{scope}
\draw[thick] (0,3) rectangle (1,4);
\draw[thick] (2,0) rectangle (3,1);
\draw[thick] (3,0) rectangle (4,1);
\draw[dashed] (0,0) rectangle (1,1);
\draw[dashed] (0,1) rectangle (1,2);
\draw[dashed] (1,0) rectangle (2,1);
\draw[dashed] (0,2) rectangle (1,3);
\draw (0.5,2.5) node {\footnotesize $b$};
\draw (0.5,3.5) node {\footnotesize $A$};
\draw (2.5,0.5) node {\footnotesize $B$};
\draw (3.5,0.5) node {\footnotesize $C$};
\end{scope}
\end{tikzpicture} 
\\ [2ex] 
\begin{tikzpicture}[scale=0.32]
\draw (0,1) node {$T'$};
\draw (0,0) node {};
\end{tikzpicture}
&
\begin{tikzpicture}[scale=0.32]
\begin{scope}
\draw[thick] (1,1) rectangle (2,2);
\draw[thick] (2,0) rectangle (3,1);
\draw[thick] (3,0) rectangle (4,1);
\draw[dashed] (0,0) rectangle (1,1);
\draw[dashed] (0,1) rectangle (1,2);
\draw[dashed] (1,0) rectangle (2,1);
\draw (0.5,2.5) node {$d$};
\draw (1.5,1.5) node {\footnotesize $A$};
\draw (2.5,0.5) node {\footnotesize $B$};
\draw (3.5,0.5) node {\footnotesize $C$};
\end{scope}
\end{tikzpicture}
& 
\begin{tikzpicture}[scale=0.32]
\begin{scope}
\draw[thick] (4,0) rectangle (5,1);
\draw[thick] (2,0) rectangle (3,1);
\draw[thick] (3,0) rectangle (4,1);
\draw[dashed] (0,0) rectangle (1,1);
\draw[dashed] (0,1) rectangle (1,2);
\draw[dashed] (1,0) rectangle (2,1);
\draw (0.5,2.5) node {$d$};
\draw (2.5,0.5) node {\footnotesize $A$};
\draw (3.5,0.5) node {\footnotesize $B$};
\draw (4.5,0.5) node {\footnotesize $C$};
\end{scope}
\end{tikzpicture}
& 
\begin{tikzpicture}[scale=0.32]
\begin{scope}
\draw[thick] (0,3) rectangle (1,4);
\draw[thick] (0,2) rectangle (1,3);
\draw[dashed] (0,0) rectangle (1,1);
\draw[dashed] (0,1) rectangle (1,2);
\draw[dashed] (1,0) rectangle (2,1);
\draw[thick]  (2,0) rectangle (3,1);
\draw (2.5,0.5) node {\footnotesize $C$};
\draw (0.5,2.5) node {\footnotesize $A$};
\draw (0.5,3.5) node {\footnotesize $B$};
\draw (3.5,0.5) node {$d$};
\end{scope}
\end{tikzpicture}
& 
\begin{tikzpicture}[scale=0.32]
\begin{scope}
\draw[thick] (0,2) rectangle (1,3);
\draw[thick] (1,1) rectangle (2,2);
\draw[dashed] (0,0) rectangle (1,1);
\draw[dashed] (0,1) rectangle (1,2);
\draw[dashed] (1,0) rectangle (2,1);
\draw[thick]  (2,0) rectangle (3,1);
\draw (2.5,0.5) node {\footnotesize $C$};
\draw (0.5,2.5) node {\footnotesize $A$};
\draw (1.5,1.5) node {\footnotesize $B$};
\draw (3.5,0.5) node {$d$};
\end{scope}
\end{tikzpicture}
& 
\begin{tikzpicture}[scale=0.32]
\begin{scope}
\draw[thick] (3,0) rectangle (4,1);
\draw[thick] (0,2) rectangle (1,3);
\draw[dashed] (0,0) rectangle (1,1);
\draw[dashed] (0,1) rectangle (1,2);
\draw[dashed] (1,0) rectangle (2,1);
\draw[thick] (2,0) rectangle (3,1);
\draw (2.5,0.5) node {\footnotesize $B$};
\draw (0.5,2.5) node {\footnotesize $A$};
\draw (3.5,0.5) node {\footnotesize $C$};
\draw (4.5,0.5) node {\footnotesize $c$};
\end{scope}
\end{tikzpicture}
& 
\begin{tikzpicture}[scale=0.32]
\begin{scope}
\draw[thick] (2,0) rectangle (3,1);
\draw[thick] (3,0) rectangle (4,1);
\draw[dashed] (0,0) rectangle (1,1);
\draw[dashed] (0,1) rectangle (1,2);
\draw[dashed] (1,0) rectangle (2,1);
\draw[thick]  (0,2) rectangle (1,3);
\draw (0.5,2.5) node {\footnotesize $A$};
\draw (0.5,3.5) node {\footnotesize $c$};
\draw (2.5,0.5) node {\footnotesize $B$};
\draw (3.5,0.5) node {\footnotesize $C$};
\end{scope}
\end{tikzpicture} 
\\ 
\end{tabular} 
\vspace{3ex}
\caption{The full set of skew shapes and SSYT for Example~\ref{exa:jdt}.  In the first row, the dashed boxes are those in $\delta$, and the solid boxes are those in $\gamma/\delta$.  Lowercase letters correspond to notation in the proof of \eqref{equ:tableauxmanipulation2}, while the uppercase $A$, $B$ and $C$ denote entries of the SSYT that we assume satisfy the necessary inequalities but are otherwise any positive integers.}
\label{tab:jdt}
\end{table}
\end{example}

\section*{Acknowledgments}  We thank Aaron Lauve for helpful comments.  R. Orellana is grateful for the hospitality of the University of Sevilla and IMUS. E. Briand, P. McNamara and M. Rosas are grateful for the hospitality of the University of Rennes 1 and IRMAR in summer 2014.

E.\ Briand and M.\ Rosas have been partially supported by projects MTM2010--19336, MTM2013-40455-P, FQM--333, P12--FQM--2696 and FEDER.  P.\ McNamara was partially supported by a grant from the Simons Foundation (\#245597). R.\ Orellana was partially supported by NSF Grant DMS-130512.

\bibliographystyle{plain}
\bibliography{commutators}


\end{document}